\documentclass{amsart}

\usepackage{amssymb}
\usepackage{amsmath}
\usepackage{amsthm}
\usepackage[mathscr]{eucal}

\usepackage{mathpazo}
\usepackage{xcolor}

\newtheorem{thrm}{Theorem}[section]
\newtheorem{lemma}[thrm]{Lemma}

\newtheorem{cor}[thrm]{Corollary}

\theoremstyle{definition}
\newtheorem{defn}[thrm]{Definition}

\theoremstyle{remark}
\newtheorem{remark}[thrm]{Remark}

\numberwithin{equation}{section}

\newcommand{\dbar}{$\bar{\partial}$}

\newcommand{\lre}{\mathcal{E}}

\makeatletter
\newcommand{\opnorm}{\@ifstar\@opnorms\@opnorm}
\newcommand{\@opnorms}[1]{%
  \left|\mkern-1.5mu\left|\mkern-1.5mu\left|
   #1
  \right|\mkern-1.5mu\right|\mkern-1.5mu\right|
}
\newcommand{\@opnorm}[2][]{%
  \mathopen{#1|\mkern-1.5mu#1|\mkern-1.5mu#1|}
  #2
  \mathclose{#1|\mkern-1.5mu#1|\mkern-1.5mu#1|}
} \makeatother

\usepackage{mathrsfs,latexsym,url,setspace}
\onehalfspace

\begin{document}

\title[Boundary value problems]{
	Boundary value problems of elliptic operators
 and reduction to the boundary techniques}

\author{Dariush Ehsani}

\address{
	FIZ Karlsruhe - Leibniz Institute for Information Infrastructure\\
	Department of Mathematics\\
	Franklinstr. 11\\
	D-10587 Berlin\\
	Germany}
\email{dehsani.math@gmail.com}

\subjclass[2010]{35J25, 35S15, 32W25}

\begin{abstract}
	We study properties of 
pseudodifferential operators which arise in their
 use in boundary value problems.  
  Smooth domains as well as intersections of 
smooth domains are considered. 
\end{abstract}

\maketitle

\bibliographystyle{plain}

\section{Introduction}

The purpose of this article is twofold.  First,
in the case of smoothly bounded
 domains, we collect and
expand on some of the known machinery involved in
the technique of reducing a boundary value problem
 to the boundary.  Second, we study some of the
operators involved in the technique in the case of 
 intersection domains.  In this case, we 
 introduce some
weighted Sobolev spaces which are of use in 
 concluding estimates.  
 
In the case of smooth domains, we follow
 the work of H\"{o}rmander \cite{Hor66}
  to handle boundary value problems
of elliptic equations in Section
 \ref{secPseudohalf}.  We provide our own
 proofs of the mapping properties of operators 
we study, although several of the proofs can be 
 adapted from those in 
  \cite{BdM71} or
  \cite{CP} or \cite{Hor66}.
   One advantage in providing our own
presentation is the relaxation of some of the 
 assumptions in the above classical works.  For 
instance, we do not consider only symbols which
 are rational as in 
  \cite{CP} or \cite{Hor66}, or even 
  symbols with poles given by homogeneous
 first order tangential symbols,
 as in \cite{BdM71}.  Another useful advantage is 
  the immediate recognition of the inverses to elliptic 
operators we consider in our examples as belonging to
 the class of operators studied, without further 
  reductions or expansions.
 
  It is the case, however,
  that in our discussion of 
  smooth domains we work with operators which 
are similar to, or can be reduced to, 
 those of Boutet de Monvel in \cite{BdM71}.
   We use operators 
  which we define
 to be {\it decomposable} 
(see Definition \ref{defnDecomp}) which, roughly
 speaking,
 defines operators with symbols which are meromorphic
with respect to the transform variable dual to the
 defining function, has residues in this variable 
which themselves are 
 symbols (in the tangential directions), as well 
  as poles whose imaginary parts are 
elliptic symbols in the tangential 
 directions.  
 We recall the definition of the 
{\it transmission property} given in \cite{BdM71}.  
We let $(x,\rho)\in \mathbb{R}^{n+1}$
  be coordinates near a boundary 
point (taken to be the origin) of 
 a smooth domain, where $\rho$ is a
 defining function for the domain.  Then with 
$\eta$ the transform variable dual to $\rho$ and
 $\xi\in \mathbb{R}^n$, the dual to $x\in\mathbb{R}^n$,
an operator of order $k$ has the transmission property if
 its symbol (and its derivatives with respect to
  the $x$ and $\rho$ variables) has an expansion
of the form
\begin{equation*}
 \sum_{j=0}^k \alpha_j(x,\xi) \eta^j +
   \sum_{j=0}^{\infty} \beta_j(x,\xi)
    \frac{(|\xi| - i \eta)^j}
    {(|\xi| + i \eta)^{j+1}},
\end{equation*}
for $\rho =0$,
with $\alpha_j$ a symbol of order $k-j$ and
 $\beta_j$  a symbol of order $k+1$,
modulo smoothing operators.  Our definition of 
 decomposable allows for inhomogeneous poles in the
  denominators, for instance a symbol of the form
\begin{equation*}
 \frac{1}{(\eta-i|\xi|b(x,\xi))^{2}}
\end{equation*}
for some (non-vanishing) 
 zero order symbol, $b(x,\xi)$.  It is not, 
 however, for this (slight) increase in generality that
  we introduce our definition, but rather because
it is immediate that an operator falls under
 our definition just by looking at its poles, without 
  first having to apply contour integrations
 or a partial fractions decomposition to see
if it fits the transmission property.  
 For instance, it is immediately seen that 
 the inverses to the elliptic operators we consider
satisfy our definition of decomposable, even if all the 
 analysis with some reduction work could be handled
  by looking at the mapping properties of 
  operators with the transmission property. 
 Our definition also allows use to treat together 
  both Poisson operators 
  and Green operators, as
 defined in \cite{BdM71}, to handle operators acting on 
  boundary distributions and distributions supported on the
 entire domain, respectively.  

As an example, consider the Laplacian,
 $\partial_x^2
  + \partial_\rho^2$ on $\mathbb{R}^2$ with
symbol $\xi^2+\eta^2$ and inverse with symbol
\begin{equation}
\label{inverseBasicLap}
 \frac{1}{\eta^2+\xi^2}
   = \frac{1}{(\eta+i|\xi|)(\eta-i|\xi|)}.
\end{equation}
When operating on a distribution
 with compact support on $\rho=0$, $f\in
\lre'(\mathbb{R}) $, we have
 for $\rho>0$
\begin{equation*}
\int \frac{\widehat{ f}(\xi)}
{(\eta+i|\xi|)(\eta-i|\xi|)} 
   e^{ix\xi}e^{i\rho\eta}
   d\xi d\eta 
    = \pi \int \frac{\widehat{ f}(\xi)}
    {|\xi|} 
    e^{ix\xi}e^{-\rho|\xi|}
    d\xi,
\end{equation*}
ignoring the singularity at $\xi=0$ (we discuss
 this in detail in Section \ref{secPseudohalf}).  Thus,
the inverse to the Laplacian 
acting on such a distribution,
 $f(x)$, has the same behavior as the operator with 
  symbol given by
\begin{equation*}
 \frac{1}
 {2|\xi|} \frac{1}
 {(\eta-i|\xi|)}.
\end{equation*}
The second factor is seen to have the transmission 
 property, whereas it was obvious from the beginning
the symbol in \eqref{inverseBasicLap} satisfies our
 condition that the poles ($\pm i|\xi|$) are elliptic
  operators in the $\partial_x$ direction,
   ignoring the singularity at
 $\xi=0$ which can be handled by multiplying with
functions which vanish identically in a neighborhood 
 of $\xi=0$.

We mention here {that} another approach to 
 boundary value problems 
 is outlined in \cite{Tr}. The approach there
is to factor an elliptic equation with each factor 
containing a normal derivative and a tangential
(pseudodifferential) operator.  
 This allows for some simplifications, in particular,
  in the calculation of the Dirichlet to Neumann
   operators.  However, the factorization approach 
 is not easily generalized to intersection domains,
which we take up in later sections of the article. 

It should be noted that we are not interested in 
 developing a full calculus for solving
boundary value problems on smoothly 
 bounded domains, as in \cite{BdM71}, as 
 the intended use of the article for the author is 
  the application of reduction to the boundary 
techniques to boundary 
 value problems which reduce to non-elliptic
  boundary conditions, as in \cite{Hor66}
   (see also \cite{CNS92} for a discussion of 
the \dbar-Neumann problem,
 the type of problem for which the analysis described 
 here is intended).  
{Thus we do not investigate invertibility or Fredholm properties of operators, 
 but rather our focus is to examine operators which arise in 
 a reduction to the boundary.

The main goal of this article is the investigation of 
 operators which arise in the reduction to the boundary techniques 
in the case of intersection domains.  As we shall see many of the 
 operators and the proofs of their properties can be derived from 
the smooth case, while some operators (the $\lre_{-\alpha}^{jk}$ operators
 in Section \ref{secInterHalf}) have no analogues in the smooth
  case.  We setup the required machinery here
   so that, for instance, properties of a solution can be
 obtained from a 
 microlocal analysis on the boundary.
We provide an example calculation in Section \ref{secEx},
 and the full force of the results outlined here will be
seen in \cite{Eh18_pwSmth}, where this exact scenario is to 
 be played out.  
}
 
{ 
The properties of operators we derive in this 
 article are aimed at a description of the regularity
of solutions to boundary value problems in terms of 
 Sobolev spaces.  In the smooth case, these are
classic results (see the above mentioned references 
 handling the smooth case).  In the case of intersections
 (in fact, on general Lipchitz domains),
regularity results have been obtained
by applying 
 layer potentials and singular integral operator
  theory for
solutions to elliptic boundary values problems, as in
the work of Jerison and Kenig
\cite{JK95}, and of Verchota \cite{Ve84}.  Thus, for instance, 
 from \cite{JK95} (see also \cite{Sh05} and
 \cite{Ve84}), 
\begin{thrm}
	Let $\Omega$ be a bounded domain with
Lipschitz boundary.  Let $P$ denote the
 Poisson operator attached to $\Omega$,
with the property
 $P(u_b) \rightarrow u_b$ almost 
everywhere, where the limits are taken
non-tangentially. 

For 
 $u_b\in W^{s}(\partial\Omega)$,
for $0\le s\le 1$,
we have
\begin{equation*}
 \| P(u_b) \|_{W^{s+1/2}(\Omega)}
  \lesssim \|u_b\|_{W^s(\Omega)}.
\end{equation*}
\end{thrm}

Our analysis in this article allows us to 
(partially) reproduce the 
 results in the above theorem (in the case
  of intersection domains), but 
also with some additional information.  
 The spaces with which we work are weighted Sobolev spaces. 
Let $\Omega \subset \mathbb{R}^n$,
  $\Omega = \cap_{j =1}^m \Omega_j$, where
 each $\Omega_j$ is a smoothly bounded domain,
and $j\le n$.  We also let 
 $\rho_j $ be the defining function for 
$\Omega_j$, and 
$\rho = \rho_1 \rho_2 \cdots \rho_m$. 

We denote by
\begin{equation*}
W^{\alpha,s}(\Omega,\rho)
 := \{ f\in W^{\alpha}(\Omega) | \rho^{r} f \in W^{\alpha+r}(\Omega), \forall\ 0\le r \le s \}
\end{equation*}
 with integer $s$ and norm
\begin{equation*} 
 \| f\|_{W^{\alpha,s}(\Omega,\rho)} 
   = \sum_{r=0}^s \| \rho^r f \|_{W^{\alpha+r}(\Omega)}.
\end{equation*}

A consequence of the techniques developed here, we
 can show
\begin{thrm}
Let $\Omega = \cap_{j =1}^m \Omega_j$ as above. 
 Let $P$ denote the Poisson operator for
$\Omega$.  Let $0 \le \alpha<1/2$, and 
 $u_{bj} \in W^{\alpha,s}(\partial\Omega_j)$.
   With $u_b \in L^2(\partial\Omega)$ with
$u_b|_{\partial\Omega_j} = u_{bj}$, we have
\begin{equation*}
\left\|
P(u_b) \right\|_{W^{\alpha+1/2,s}
	\left( \Omega, \rho
	\right)}
\lesssim \sum_j \left\|
u_{bj} 
\right\|_{W^{\alpha,s}
	\left(
	\partial\Omega_j\cap \partial\Omega,
	\rho_{\hat{j}} 
	\right) },
\end{equation*}
where $\rho_{\hat{j}}: = \rho_1 \cdots \rho_{j-1} \rho_{j+1}\cdots \rho_m$.

\end{thrm}
This is a special case of what is
 proved in \cite{Eh18_pwSmth}.
We emphasize here that with our techniques we not only 
 obtain the known $1/2$ gain of regularity, we also 
obtain information on the degree of singularity of 
the solution, on how the regularity or singularity is 
affected upon applying derivatives.  This is a subtle issue
 in the case of the \dbar-Neumann problem (see
 \cite{Eh18_pwSmth}  and \cite{EhLi12}).

}
 
{
 Section \ref{secPseudohalf} is somewhat expository,
  reviewing some of the classical results of 
 psuedodifferential operators arising from boundary value problems
on smooth domains.  There are many 
 works treating the smooth case and we refer the 
 interested reader to \cite{BdM71}, \cite{BrSh}, \cite{Es}, \cite{Gru},
 \cite{Hor66}, \cite{Schu}, and \cite{Tr}
 for detailed treatments of 
  boundary value problems on smooth domains.
\cite{Schu} in addition handles some
 non-smooth cases.  
We provide our own proofs of the operators which
we study and which arise in
 boundary value problems, which will,
in addition to providing the above mentioned
simplifications, serve as a preparation for the operators
 which are to arise in the case of intersections, in some cases
serving as base cases to which operators in the intersection case
 will be reduced.}

 In Section \ref{secInterHalf}, 
the case of (transversal) intersections of domains is
  considered, with the main results
relating to extending estimates 
 obtained in the case of smooth domains via 
weighted estimates.  As far as the author is aware,
 the technique of reducing to the boundary in the
  case of intersection domains has not been extensively
studied.  One reason is certainly some of the troubling
 operators which mix distributions of the 
  different boundaries, which do not behave as 
pseudodifferential operators.  Nonetheless, several
 pseudodifferential operators do arise, and weighted
Sobolev spaces are defined and then those operators are
 studied with respect to the weighted spaces.  
The above-mentioned boundary value operators which 
 are not pseudodifferential operators are also 
  studied with respect to the weighted spaces.  

In Section \ref{secEx} we illustrate how the
 various operators presented arise in the 
  technique of reduction to the boundary 
on intersection domains, in the simple case of
 an intersection of two smoothly bounded domains. 
We focus both on the boundary value operators which
 arise as well as those conditions which 
determine the 
  ellipticity of the boundary conditions 
 (with elliptic highest order operators).
The application of the methods and
properties of Section \ref{secInterHalf}
 to the \dbar-Neumann problem, whose boundary
  conditions reduce to non-elliptic problems 
is the subject of a current study of the author.

With the use of partitions of unity and cutoffs
 (in a neighborhood of a boundary 
 point of a smooth domain), 
 we can
assume coordinates $(x,\rho)\in \mathbb{R}^{n+1}$ for $\rho<0$ and
 thus reduce the study to the case of 
operators acting on distributions supported in the
 lower-half {space} or distributions supported in 
  $\mathbb{R}^n$.  In the case of intersections, the
 coordinates will be chosen so that near a point 
on the intersection of several boundaries, the domain
 looks like the intersection of several 
lower-half {spaces}.

The author gives warm thanks to S\"{o}nmez \c{S}ahuto\u{g}lu
 for helpful comments and suggestions on drafts of this
work.  {The author also acknowledges and gives thanks for the 
 many helpful comments from the referee. }

\section{Analysis on the lower half-space}
\label{secPseudohalf} 
In this section we develop some of the
properties of pseudodifferential operators on half-spaces.  Of
particular importance for the reduction to the boundary techniques
are the boundary values of pseudodifferential operators on
half-spaces, pseudodifferential operators acting on
distributions supported on the boundary, as well as
 pseudodifferential operators on the boundary itself.
 
We briefly recall the definition of pseudodifferential
operators in $\mathbb{R}^n$.  We refer to the book
by Treves, \cite{Tr}, for a thorough introduction to the subject.
Let $\mathcal{S}^m(\mathbb{R}^n)$ denote the 
space of {\it symbols} in $\mathbb{R}^n$.  
A symbol, $a(x,\xi)\in \mathcal{S}^m(\mathbb{R}^n)$
is a
$C^{\infty}(\mathbb{R}^n \times
\mathbb{R}^n)$ function with estimates
on any compact $K \subset \mathbb{R}^n$
\begin{equation}
\label{defnPsdo}
|  \partial_{\xi}^{\alpha}\partial_x^{\beta}
a(x,\xi)| \le c_{\alpha,\beta}(K) 
(1+|\xi|)^{m-|\alpha|} \qquad 
\forall x\in K, \xi\in \mathbb{R}^n,
\end{equation}
where $c_{\alpha,\beta}(K) > 0$.
Note that, for a multi-index, $\beta=(\beta_1, \beta_2,\ldots,\beta_n)$, 
 we make use of the index notation
\begin{equation*}
\partial_x^{\beta} 
= \partial_{x_1}^{\beta_1}
\cdots \partial_{x_n}^{\beta_n}
\end{equation*}
for both $x$ and $\xi$ derivatives.

A pseudodifferential operator
of class $\Psi^m(\mathbb{R}^n)$ is defined in 
terms of a symbol of class $\mathcal{S}^m(\mathbb{R}^n)$.
A pseudodifferential operator, 
$A\in\Psi^m(\mathbb{R}^n)$, can be written as
\begin{equation*}
Af = \frac{1}{(2\pi)^n}
\int a(x,\xi) \widehat{f}(\xi) e^{ix\xi} d\xi.
\end{equation*}
We write $A=Op(a)$.
Generally, we will consider $f$ to be a function
in a Sobolev space, $W^s(\mathbb{R}^n)$, defined as 
the space of functions
such that 
$(1+|\xi|)^{s/2}|\hat{f}(\xi)| \in L^2(\mathbb{R}^n)$.
We have that $A:W^{s}(\mathbb{R}^n) 
\rightarrow W^{s-m}(\mathbb{R}^n)$
as an operator between Sobolev spaces.

An {\it elliptic} operator is a particular type
of pseudodifferential operator whose symbol,
 $a(x,\xi)$, is such that there exist
positive functions $c(x)$ and $r(x)$, and
for each $x$
\begin{equation}
\label{defnEll}
c(x) |\xi|^m \le |a(x,\xi)| 
\qquad \forall\ |\xi| \ge r(x)
\end{equation}
(for $a(x,\xi)\in \mathcal{S}^m
(\mathbb{R}^n)$).
Then we say $Op(a)$ is an elliptic operator,
see \cite{Tr}.

We now fix some notation to be used
throughout this paper.
 We will reserve the notation, $A_k$, to refer to
some pseudodifferential operator of order $k$;
{ thus by $A_k$, we mean $A_k\in \Psi^k(\mathbb{R}^n)$}
(or in the space of operators defined on a 
given domain, $\Omega$, depending on context).
We will even allow the specific operator
referred to by $A_k$ to change from one line 
to the next, {or within the same line itself}.  Thus for instance
we can write $A_0 \circ A_1 = A_1$.

We use coordinates $(x,\rho)$ on $\mathbb{R}^{n+1}$
with $x = (x_1,\ldots, x_n)\in \mathbb{R}^n$.
The full Fourier Transform of a 
function, $f(x,\rho)$, will be written
\begin{equation*}
\widehat{f}(\xi,\eta) = 
\frac{1}{(2\pi)^{n+1}} \int
f(x,\rho) e^{-ix\xi}e^{-i\rho\eta} dx d\rho
\end{equation*}
where $\xi=(\xi_1,\ldots,\xi_n)$.  For $f$ defined on 
a subset of $\mathbb{R}^{n+1}$ we define its Fourier
Transform as the transform of the function extended
by zero to all of $\mathbb{R}^{n+1}$.  Thus,
for instance for $f$ defined on 
$\{(x,\rho)\in
\mathbb{R}^{n+1}: \rho<0\}$, we write
\begin{equation*}
\widehat{f}(\xi,\eta) = 
\frac{1}{(2\pi)^{n+1}} 
\int_{-\infty}^0 \int_{\mathbb{R}^n}
f(x,\rho) e^{-ix\xi}e^{-i\rho\eta} dx d\rho.
\end{equation*}

A partial Fourier Transform in the $x$ variables
will be denoted by
\begin{equation*}
\widetilde{f}(\xi,\rho) = 
\frac{1}{(2\pi)^{n}} \int_{\mathbb{R}^n}
f(x,\rho) e^{-ix\xi} dx .
\end{equation*}

We define the half-space $\mathbb{H}_-^{n+1}:= \{(x,\rho)\in
\mathbb{R}^{n+1}: \rho<0\}$.
The space of distributions,
$\mathscr{E}'(\overline{\mathbb{H}}^{n+1}_-)$ is defined as the
compactly supported distributions in $\mathscr{E}'(\mathbb{R}^{n+1})$
with support in $\overline{\mathbb{H}}^{n+1}_-$.  The topology
of $\mathscr{E}'(\overline{\mathbb{H}}^{n+1}_-)$ is inherited
from that of $\mathscr{E}'(\mathbb{R}^{n+1})$.
We endow $C^{\infty}(\mathbb{R}^{n+1})$ with the topology defined
in terms of the semi-norms
\begin{equation*}
p_{l,K}(\phi)=\max_{(x,\rho)\in K \subset\subset\mathbb{R}^{m+1}}
\sum_{|\alpha|\le l} |\partial^{\alpha} \phi(x,\rho)|.
\end{equation*}
A {\it regularizing} operator,
 $A \in \Psi^{-\infty}(\mathbb{R}^{n+1})$,
is a continuous
linear map
\begin{equation}
\label{regop}
A:
\mathscr{E}'(\mathbb{R}^{n+1}) \rightarrow
C^{\infty}(\mathbb{R}^{n+1}).
\end{equation}
{We will use the term {\it smoothing}
 to describe restrictions of operators 
as long as smoothness and continuity properties are
 exhibited.  For instance, we say an operator
$A:
\mathscr{E}'(\mathbb{R}^{n})\times \delta(\rho) \rightarrow
C^{\infty}(\mathbb{R}^{n+1})$  
is smoothing on $\mathscr{E}'(\mathbb{R}^{n})\times \delta(\rho)$
 if it is a continuous and 
 linear.  Note that in this case it may not be
 true that $A\in \Psi^{-\infty}(\mathbb{R}^{n+1})$. }

In working with pseudodifferential operators on half-spaces, with
coordinates $(x_1,\ldots, x_n, \rho)$, $\rho<0$, we
can 
show that by multiplying symbols 
 by smooth cutoffs with compact support, which are
functions of transform variables corresponding to tangential
coordinates, we produce operators which are smoothing.
This is in analogy to the case of $\mathbb{R}^n$ where
cutoffs (with compact support) in transform space
give rise to regularizing operators:
let $a(x,\xi) \in \mathcal{S}^k(\mathbb{R}^n)$,
and 
$\chi(\xi) \in C^{\infty}_0(\mathbb{R}^n)$.
Then the operator with 
symbol $\chi(\xi) a(x,\xi)$ is regularizing.  
The reason behind this is that any growth
in the $\xi$ variables resulting from differentiation
is compensated by the compact support of 
$\chi(\xi)$.  For instance,
for $x$ in a compact subset $K$,
\begin{align*}
\left| 
\partial_x^{\alpha}
\int a(x,\xi) \chi(\xi) \hat{\phi}(\xi)
e^{ix\xi} d\xi
\right|
\lesssim& \sum_{\alpha_1 + \alpha_2 = \alpha}
\int    \left|\partial_x^{\alpha_1} a(x,\xi)
\chi(\xi)   \xi^{\alpha_2}
\hat{\phi}(\xi)    \right|d\xi
\\
\lesssim& 
\int
|\chi(\xi)| (1+|\xi|)^{|\alpha|+k}
| \hat{\phi}(\xi)|
d\xi\\
\lesssim& \|\phi\|_{L^2(\mathbb{R}^n)}  
\end{align*}
for all {$\alpha = (\alpha_1, \ldots, \alpha_n$) with} $\alpha_i \ge 0$,
where the constants of inequality depend only
on the compact set $K$.

\begin{lemma}
	\label{symbolcut}
	Let $A\in \Psi^{-k}(\mathbb{R}^{n+1})$,
	for $k\ge 1$ an integer, be such that the symbol,
	$\sigma(A)(x,\rho,\xi,\eta)$, 
	is meromorphic (in $\eta$) with poles at 
	\begin{equation*}
	\eta=q_1(x,\rho,\xi), \ldots, q_k(x,\rho,\xi)
	\end{equation*}	
	with $q_i(x,\rho,\xi)$ themselves symbols of 
	pseudodifferential operators of
	order 1 (restricted to $\eta=0$)
	such that for each 
	$\rho$, $\mbox{Res}_{\eta=q_i}
	\sigma(A) \in \mathcal{S}^{-k+1}
	(\mathbb{R}^n)$ with symbol estimates
	uniform in the $\rho$ parameter.

	Let $A_{\chi}$ denote the operator with symbol
	\begin{equation*}
	\chi(\xi)\sigma(A),
	\end{equation*}
	where $\chi(\xi)\in C_0^{\infty}(\mathbb{R}^{n})$.  Then
	$A_{\chi}$
	is smoothing on distributions supported on the boundary:
	\begin{equation*}
	A_{\chi}:
	\mathscr{E}'(\mathbb{R}^{n}) \times \delta(\rho)
	\rightarrow C^{\infty}(\overline{\mathbb{H}}_-^{n+1}),
	\end{equation*}
where $\delta$ is the Dirac-delta distribution.
\end{lemma}
\begin{proof}
	Without loss of generality we suppose $\phi_b(x)\in
	L^2_c(\mathbb{R}^{n})$, and let $\phi=\phi_b\times\delta(\rho)$.
	We estimate derivatives of
	$ A_{\chi}(\phi)$.  We let $a(x,\rho,\xi,\eta)$ denote the symbol of
	$A$, and $a_{\chi}(x,\rho,\xi,\eta)$ that of $A_{\chi}$.
	
We use induction on the 
(absolute value of the) order, $k$, of the 
operator.  The base case $k=1$ will follow from the
calculations of the induction step.  For a given
$k$, we thus assume the Lemma has been proven for
operators in $\Psi^{-1}(\mathbb{R}^{n+1})$,
$ \ldots$, $\Psi^{-(k-1)}(\mathbb{R}^{n+1})$.

	We first note that
	derivatives with respect to the $x$ variables
	pose no difficulty due to the $\chi$ term in the symbol of
	$A_{\chi}$:
	from
	\begin{align*}
	A_{\chi} \phi=&
	\frac{1}{(2\pi)^{n+1}}\int
	a_{\chi}(x,\rho,\xi,\eta)
	\widehat{\phi}(\xi,\eta) e^{ix\xi}e^{i\rho\eta}
	d\xi d\eta\\
	=&
	\frac{1}{(2\pi)^{n+1}} \int
	a_{\chi}(x,\rho,\xi,\eta)
	\widetilde{\phi}_b(\xi) e^{ix\xi}e^{i\rho\eta}
	d\xi d\eta,
	\end{align*}
	we calculate over $(x,\rho) \in K \subset\subset 
	\mathbb{R}^{n+1}$
	\begin{align*}
	\left| \partial^{\alpha}_x
	A_{\chi} \phi\right| &\lesssim
	\sum_{\alpha_1+\alpha_2 = \alpha}
	\int
	| \partial^{\alpha_1}_x a_{\chi}(x,\rho,\xi,\eta)|
	|\xi|^{|\alpha_2|}\left|\widetilde{\phi}_b(\xi)\right|
	d\xi d\eta\\
	&\lesssim \|\phi_b\|_{L^2} 
	\sum_{\alpha_1+\alpha_2 = \alpha} \left( \int
	|\xi|^{|2\alpha_2|}\chi^2(\xi) | \partial^{\alpha_1}_x
	a(x,\rho,\xi,\eta)|^2
	d\xi d\eta\right)^{1/2}\\
	&\lesssim \|\phi_b\|_{L^2} \left( \int
	|\xi|^{|2\alpha|}\chi^2(\xi)
	\frac{1}{(1+\xi^2+\eta^2)^{
			k}}
	d\xi d\eta\right)^{1/2}\\
	&\lesssim \|\phi_b\|_{L^2} \left( \int
	|\xi|^{|2\alpha|}\chi^2(\xi)
	d\xi\right)^{1/2} \\
	&\lesssim
	\|\phi_b\|_{L^2}
	,
	\end{align*}
	where the constants of inequalities depend
	on the compact, $K$.   For any mixed derivative
	$\partial_x^{\alpha}\partial_{\rho}^{\beta}$, the $x$ derivatives
	can be handled in the manner above and so we turn to derivatives
	of the type $\partial^{\beta}_{\rho} A_{\chi} \phi $.

	We use the residue calculus to integrate over the $\eta$ variable
	in
	\begin{equation*}
	A_{\chi} \phi=
	\frac{1}{(2\pi)^{n+1}}\int
	a_{\chi}(x,\rho,\xi,\eta)
	\widetilde{\phi}_b(\xi) e^{ix\xi}e^{i\rho\eta}
	d\xi d\eta.
	\end{equation*}
	Denote the poles of 
	$a(x,\rho,\xi,\eta)$ which are in the lower half-{space}
	($\mbox{Im }\eta <0$) by
	\begin{equation*}
	\eta = q_1^-(x,\rho,\xi), \ldots,
	q_l^-(x,\rho,\xi),
	\end{equation*}
	where $l\le k$.
	Let
	\begin{equation*}
	a_{q_j^-}(x,\rho,\xi) =  i
	\mbox{Res}_{\eta=q_j^-} a_{\chi}(x,\rho,\xi,\eta)
	\end{equation*}
	and
	\begin{equation*}
	A_{\chi}^{q_j^-} \phi=
	\frac{1}{(2\pi)^{n}} \int
	a_{q_j^-}(x,\rho,\xi)
	\widetilde{\phi}_b(\xi) e^{ix\xi}
	e^{i\rho q_j^-}  d\xi.
	\end{equation*}
	
	In particular, we have
	\begin{equation*}
	A_{\chi}\phi 
	=   \sum_{1\le j\le l} A_{\chi}^{q_j^-} \phi,
	\end{equation*}	
modulo terms which are handled by the induction 
 hypothesis, which arise in the case of 
poles of order higher than one (namely, 
 from the resulting $\eta$ derivatives landing
on the $e^{i\rho\eta}$ term).  
Thus estimates (of derivatives) of 
	$A_{\chi}\phi$ will be deduced from estimates of
	$A_{\chi}^{q_j^-} \phi$.  We note {that while} the above sum 
in terms of operators with symbols as the residues of 
 $\sigma(A_{\chi})$ is also valid in the case of poles
  of multiplicity higher than one,
  the hypothesis that the residues are themselves symbols 
 excludes the case in which two poles merge at a point or 
  neighborhood in the domain but are not identical.
	
We can now estimate $\partial^{\beta}_{\rho}
A_{\chi}^{q_j^-} \phi $, for some 
 $j$,
by differentiating under the integral.  
	Note that the factor of $\chi(\xi)$ is contained in 
	$a_{q_j^-}$.  A term
 $\left|\partial_{\rho}^{\beta}\left( a_{q_j^-} e^{i\rho q_j^-}\right)\right|$
   is bounded by a sum of
terms of the form
\begin{equation*}
 s_{\beta_1\beta_2\alpha} =   
 \left|
 	\partial_{\rho}^{\beta_1} a_{q_j^-} 
 \right|
   		(1+|q_j^-|)^{\alpha_0}
   		 \left|(\partial_{\rho}q_j^-)^{\alpha_1}
   		 \cdots
   (\partial_{\rho}^{\beta_2}q_j^-)^{\alpha_{\beta_2}}
   \right|
\end{equation*}
with $\alpha = (\alpha_0,\alpha_1,\ldots,\alpha_{\beta_2})$ and
\begin{equation}
\label{sumCondn}
\begin{aligned}
 & \beta_1 + \beta_2 = \beta,\\
 &\alpha_0 + \sum_{j=1}^{\beta_2} j \cdot \alpha_j = \beta_2.
\end{aligned}
\end{equation}
From the properties of a symbol, we can estimate
\begin{equation*}
 |s_{\beta_1\beta_2\alpha}| 
   \lesssim (1+|\xi|)^{-k+1+\beta_2}.
\end{equation*}
	For $\rho<0$, over $(x,\rho) \in K \subset\subset 
	\mathbb{R}^{n+1}$, we have
	the estimates
	\begin{align*}
	\left|\partial^{\beta}_{\rho} A_{\chi}^{q_-} \phi\right| &\lesssim
	\sum_{\beta_1,\beta_2,\alpha}
	\int 
	s_{\beta_1\beta_2\alpha}
	\left|\widetilde{\phi}_b(\xi)\right|
	d\xi \\
	&\lesssim
	\|\phi_b\|_{L^1}
	,
	\end{align*}
where the summation is over $\alpha,\beta_1,\beta_2$ which satisfy the 
 conditions in \eqref{sumCondn}.
	Again, the integral over $\xi$
	converges due to the factor of $\chi(\xi)$ contained in
	$a_{q_j^-}$.
	
The base case, $k=1$, follows the same calculations 
 above.  
We conclude the
	proof of the lemma.
\end{proof}

The motivation for Lemma \ref{symbolcut} is from the
consideration of inverses to elliptic operators.  We
recall that an elliptic operator, $B\in \Psi^{m}
(\mathbb{R}^{n+1})$, has an inverse, $A \in \Psi^{-m}
(\mathbb{R}^{n+1})$, such that
$B\circ A = A\circ B = I$ modulo
$\Psi^{-\infty}(\mathbb{R}^{n+1})$.  The symbol
of $A$ can be determined using the symbol calculus
and can be written as a sum of terms 
which have powers of $\sigma(B)$ in their denominators.
A standard procedure to deal with possible zeros in the denominator
 is to introduce
cutoffs, $\phi_j(x,\rho,\xi,\eta)$, which vanish 
near the zeros of $\sigma(B)$ are used.  For example,
the symbol of the inverse, $A$,
to the Laplacian operator, $B$, with symbol
$\sigma(B) = \xi^2 + \eta^2$, is given by
$\sigma(A) = \phi(\xi,\eta) /(\xi^2 + \eta^2)$,
where $\phi(\xi,\eta)$ is chosen to be 
$0$ in a small neighborhood of the origin.
The problem in applying Lemma \ref{symbolcut} to
such an inverse, $A$, is that the symbol,
$\sigma(A)$ is no longer meromorphic 
in $\eta$ due to the use of the cutoff,
$\phi(\xi,\eta)$.  {We therefore mention to the reader
 interested in applying our analysis that
a technique which can be applied to the situation
 of Lemma is to add a zero order operator to
$B$ and consider instead the inverse to $B+B_0$,
 where $B_0$ is chosen such that $\sigma(B+B_0)$ does
  not vanish.  This technique has been applied to 
the case of intersection domains \cite{Eh18_pwSmth}.}

 We now prove a lemma relating to the
zeros of $\eta$, denoted $q_i(x,\rho,\xi)$ above.
\begin{lemma}
\label{lemmaZerosEllOp}
 Let $B \in \Psi^m(\mathbb{R}^{n+1})$ be an
elliptic pseudodifferential operator with
 symbol
\begin{equation}
\label{factor}
 \sigma(B) = (\eta - q_1(x,\xi,\rho)) \cdots (\eta - q_m(x,\xi,\rho)),
\end{equation}
a polynomial of order $m$ in the $\eta$ variables.  Then
 the $q_i(x,\rho,\xi)$ are {symbols of} elliptic operators 
 of order one (with $\rho$ as a parameter).
\end{lemma}
\begin{proof}
We fix $x$ to be in some compact set.  We want to show
 that each $q_i$ satisfies 
$|q_i| \simeq |\xi|$ for large $\xi$.

Setting $\eta = 0$ in \eqref{factor}, we know 
from ellipticity
\begin{equation*}
|q_1 \cdot q_2 \cdots q_m| \simeq |\xi|^m.
\end{equation*}

If one $q_i$ (suppose it is $q_1$)
were such that $|\xi| / |q_1| \rightarrow 0$,
then with $r_{m-1}  = q_2 \cdots q_m$,
we would have $r_{m-1}$ is such that
$|r_{m-1}|/|\xi|^{m-1} \rightarrow 0$.
 We will use the notation $r_k$ below to
denote the coefficient of 
 $\eta^{m-1-k}$ in the polynomial
\begin{equation*}
  (\eta - q_2)(\eta - q_3) \cdots (\eta - q_m).
\end{equation*}

Now take $k$ derivatives with respect to
 the $\eta$ variables of the symbol, $\sigma(B)$,
and set $\eta=0$.  We obtain an inequality 
from \eqref{defnPsdo} of
the form
\begin{equation*}
|q_1| \cdot r_{m-1-k} + r_{m-k} \lesssim
|\xi|^{m-k}.
\end{equation*} 

At each step, for $k=1,\ldots, m-1$, we conclude
$|r_{m-1-k}|/|\xi|^{m-1-k} \rightarrow 0$, due to
the known (slow) growth of $r_{m-k}$, and
(fast) growth of $q_1$.

With $m-1$ derivatives, we have
\begin{equation}
 \label{sum}
|q_1 + q_2 + \cdots +q_m| \lesssim |\xi|,
\end{equation}  
but from the growth of
 $r_1$ from above, we
have the property
\begin{equation*}
|q_2+\cdots +q_m|/|\xi| \rightarrow 0.
\end{equation*}
  The 
 relation in \eqref{sum} thus leads to a contradiction
because 
$q_1$ was supposed to have {faster} growth than
$|\xi|$.
\end{proof}

The next theorem is aimed at properties of 
a finite sum
of the first terms of the expansion of
the inverse to an elliptic operator
of order $k\ge 1$.  For such operators,
the
hypotheses of Lemmas \ref{symbolcut} and \ref{lemmaZerosEllOp}
 are satisfied, and 
so the poles $\eta = q_i(x,\rho,\xi)$ are
elliptic (uniformly in the 
$\rho$ parameter) of order 1, as are their
imaginary parts.
For such operators,
without the assumption of the cutoff function
$\chi(\xi)$ in the symbol of $A_{\chi}$ we can still prove
\begin{thrm}
	\label{estinvell}
	Let $g\in {\mathscr{D}'}(\mathbb{R}^{n+1})$ of
	the
	form $g(x,\rho) =g_b(x) \delta(\rho)$;
	$g_b$ is a distribution supported on 
$\partial\mathbb{H}^{n+1}_- = \mathbb{R}^n$.  
	Let $A\in \Psi^{k}(\mathbb{R}^{n+1})$, 
	$k\le -1$ be as in Lemma \ref{symbolcut}
	with the additional assumption that
$\sigma(A)(x,\rho,\xi,\eta)$ vanishes in a neighborhood
 of $\xi = 0$ and 
	the imaginary parts of the poles,
	$q_i(x,\rho,\xi)$ are symbols of elliptic operators
	(restricted to $\eta=0$) of order 1.  Then
	for all integer $s\ge 0$,
and $g_b \in W^{s+k+1/2}(\mathbb{R}^n)$,
	\begin{equation*}
	\|\varphi A g\|_{W^s(\mathbb{H}^{n+1}_-)} 
\lesssim \|g_b\|_{W^{s+k+1/2}(\mathbb{R}^{n})}
	\end{equation*}
for any $\varphi\in C^{\infty}_0
 (\overline{\mathbb{H}}^{n+1}_-)$.	
The estimate also holds for all (non-integer)
 $s\ge |k|-1$.
\end{thrm}
\begin{proof}
	We follow and use the notation of the proof of Lemma
	\ref{symbolcut}, and
prove by induction, assuming the Theorem holds for
 operators in $\Psi^{-j}(\mathbb{R}^{n+1})$,
  for $j=1,\ldots, k-1$.

We again
	analyze a typical term resulting from a pole at
	$\eta=q_j^-(x,\rho,\xi)$
 of $a(x,\rho,\xi,\eta)$.   With
	\begin{equation*}
	a_{q_j^-}(x,\rho,\xi) = i
	\mbox{Res}_{\eta=q_j^-} a(x,\rho,\xi,\eta)
	\end{equation*}
	and
	\begin{equation*}
	A^{q_j^-} g=
	\frac{1}{(2\pi)^{n}} \int
	a_{q_j^-}(x,\rho,\xi)
	\widetilde{g}_b(\xi) e^{ix\xi}
	e^{i\rho q_j^-}  d\xi,
	\end{equation*}
	we estimate $\partial^{\alpha}_x \partial^{\beta}_{\rho} A^{q_j^-} g
	$
	in the case of integer $s\ge 0$ and
	$|\alpha| + \beta =s$ 
	by differentiating under the integral.
As in the proof of Lemma \ref{symbolcut}, we bound 
 $\partial^{\alpha}_x \partial^{\beta}_{\rho}
  \left( a_{q_j^-} e^{ix\xi}
  e^{i\rho q_j^-} \right)$ by a sum of terms
\begin{equation*}
s_{\alpha\beta\gamma} = 
(1+|\xi|)^{\alpha_1}  
\left|
\partial_x^{\alpha_2}\partial_{\rho}^{\beta_1}
  a_{q_j^-} 
\right|
(1+|q_j^-|)^{\gamma_0} \left|(D_{x,\rho}q_j^-)^{\gamma_1}
\cdots
 (D_{x,\rho}^{\alpha_3+\beta_2}
 q_j^-)^{\gamma_{\alpha_3+\beta_2}}
\right|
\end{equation*}
with
 $D_{x,\rho}^j$ a derivative of the form
  $\partial_{\rho}^i \partial_x^l$ for $i + |l| = j$,
   $\alpha = (\alpha_1,\alpha_2,\alpha_3)$,
   $\beta = (\beta_1,\beta_2)$,
 $\gamma = (\gamma_0,\gamma_1,\ldots,\gamma_{\alpha_3+\beta_2})$ and
\begin{equation}
\begin{aligned}
& \alpha_1+ \alpha_2+ \alpha_3 = \alpha,\\
& \beta_1 + \beta_2 = \beta,\\
&\gamma_0 + \sum_{j=1}^{\alpha_3+\beta_2} j \cdot \gamma_j = \alpha_3+\beta_2.
\end{aligned}
\end{equation}
We note the estimates
\begin{align*}
 | s_{\alpha\beta\gamma} | 
  \lesssim& (1+|\xi|)^{\alpha_1+k+1+\alpha_3+\beta_2}\\
  \lesssim& (1+|\xi|)^{k+1+s}.
\end{align*}

	For $\rho<0$ we thus have
	the estimates
	\begin{equation*}
	\left|\partial_x^{\alpha} \partial^{\beta}_{\rho}
	A^{q_j^-} g\right|^2 \lesssim
	\int
(1+|\xi|^2)^{k+1+s}\left|\widetilde{g}_b(\xi)\right|^2
	e^{2\rho |\mbox{Im }q_j^-|}
	d\xi
	,
	\end{equation*}
	where as in Lemma \ref{symbolcut}, the estimates are
	over compact sets, $K$, $(x,\rho) \in K \subset\subset 
	\mathbb{R}^{n+1}$, with the constants of the inequality
	depending on $K$.
	
	Since  $\mbox{Im } q_j^-$ 
	is assumed to be the symbol of an 
	order 1 elliptic operator,
	we have
	$ |\xi| \lesssim | \mbox{Im } q_j^-|
	\lesssim (1 + |\xi|)$ using the inequalities
	\eqref{defnPsdo}, \eqref{defnEll}
	(for $x$ restricted to a compact set, and the
	constants of inequalities depending on that set).
	Thus,
	we have the property
	\begin{equation*}
	\frac{1}{|\mbox{Im }q_j^-|}\sim \frac{1}{|\xi|}.
	\end{equation*}
	Thus, when we integrate over $\rho$ we get a factor
	on the order 
	of $\frac{1}{|\xi|}$ which lowers the order of 
	the	norm in the tangential directions by 1/2:
	\begin{align*}
	\left\|
	 \varphi \partial_x^{\alpha} \partial^{\beta}_{\rho}
	A^{q_j^-} g \right\|_{L^2}^2 &\lesssim  \int
(1+|\xi|^2)^{k+1+s} \left|\widetilde{g}_b(\xi)\right|^2
	\frac{1}{1+|\xi|}
	d\xi\\
	&\lesssim
	\sum_{\alpha,\beta}
	\int |1+|\xi|^2|^{k+1/2+s}
	\left|\widetilde{g}_b(\xi)\right|^2
	d\xi\\
	&\lesssim \|
	g_b\|^2_{W^{s+k+1/2}(\mathbb{R}^{n})},
	\end{align*}
	Note that we can use the term $1+|\xi|$ in the
	denominator 
by the assumption that $a(x,\rho,\xi,\eta)$ vanishes near
 $\xi= 0$.

The base case $k=-1$ is handled by the calculations
above, and
the rest of the proof, including how to 
incorporate the induction step follows that of
Lemma \ref{symbolcut}.

	This proves the Theorem for integer $s\ge 0$.
	The non-integer case follows by interpolation
\cite{LiMa}.
\end{proof}

{
\begin{remark}
In practice, the assumption of the vanishing of the symbol at
 $\xi=0$ can be removed by the consideration of symbols
  with non-vanishing
 denominators (by modifying the operators with 
  the addition of zero order terms for instance, see the
discussion after Lemma \ref{symbolcut}).
\end{remark}
}

Another version of Theorem \ref{estinvell} is
  contained in 
 \cite{CP} (see Theorem 5.2.4 $iii)$).

The hypotheses of Lemmas \ref{symbolcut} and
 \ref{lemmaZerosEllOp}, and Theorem
\ref{estinvell} are all satisfied for instance in the 
 case of the terms in the expansion of the
inverse to an elliptic differential operator
such as the Laplacian.   

There are analogue estimates for functions with 
support in the half-space (as opposed to support
on the boundary):
\begin{thrm}
\label{thrmSobIntEst}
Let $k\le -1$, $s\ge |k|$, and $f\in W^{s+k} (\mathbb{H}^{n+1}_-)$.
  Let
$A\in \Psi^{k}(\mathbb{R}^{n+1})$ be as in Theorem \ref{estinvell}.  
Then,  
\begin{equation*}
\|\varphi A f\|_{W^s(\mathbb{H}^{n+1}_-)} \lesssim \|
f\|_{W^{s+k}(\mathbb{H}^{n+1}_-)}
\end{equation*}
for any $\varphi\in C^{\infty}_0
(\overline{\mathbb{H}}^{n+1}_-)$.
\end{thrm}
\begin{proof}
{
	We  prove by
 induction on the order of the 
  class $\Psi^k(\mathbb{R}^{n+1})$.

Let $a(x,\rho,\xi,\eta)$ be
 the symbol of the operator $A$: 
\begin{equation*}
 Af=\frac{1}{(2\pi)^{n+1}} \int
 a(x,\rho,\xi,\eta)
 \widehat{f}(\xi,\eta) e^{ix\xi} e^{i\rho\eta}  d\xi d\eta,
\end{equation*}

In the case $k={-1}$ we can write
$a(x,\rho,\xi,\eta) = 
\frac{1}{\eta - q(x,\rho,\xi)}$, and calculate
\begin{align*}
 \partial_{\rho}
   Af=& \frac{i}{(2\pi)^{n+1}} \int
  \eta a(x,\rho,\xi,\eta)
  \widehat{f}(\xi,\eta) e^{ix\xi} e^{i\rho\eta}  d\xi d\eta\\
=& \frac{i}{(2\pi)^{n+1}} \int
(\eta-q(x,\rho,\xi)) a(x,\rho,\xi,\eta)
\widehat{f}(\xi,\eta) e^{ix\xi} e^{i\rho\eta}  d\xi d\eta\\
& \quad + \frac{i}{(2\pi)^{n+1}} \int
a(x,\rho,\xi,\eta)
q(x,\rho,\xi) \widehat{f}(\xi,\eta) e^{ix\xi} e^{i\rho\eta}  d\xi d\eta\\
=& if(x,\rho) +  \frac{i}{(2\pi)^{n+1}} \int
a(x,\rho,\xi,\eta)
q(x,\rho,\xi) \widehat{f}(\xi,\eta) e^{ix\xi} e^{i\rho\eta}  d\xi d\eta.
\end{align*}
In this way, we can relate $\rho$ derivatives to 
 derivatives in the tangential directions. 
The second term on the right can be estimated by
 $\|f\|_{L^2(\mathbb{H}^{n+1}_-)}$.  This can be repeated to 
  show any $\rho$ derivative of order $s$ can be 
estimated by $\|f\|_{W^{s-1}(\mathbb{H}^{n+1}_-)}$.
 Estimates for derivatives with respect to $x$ are handled 
directly as with the second term above and the base case of 
 $k=-1$ is proved.

Lower order operators (for $k{<-1}$) are handled similarly.
 Let $k<-1$.  
  We let
 $q_1, \ldots , q_{|k|}$ 
 denote the poles of
 $a(x,\rho,\xi,\eta)$ (in $
 \eta$)  counted with
 multiplicity.  We have 
\begin{align*}
\partial_{\rho}
Af=& \frac{i}{(2\pi)^{n+1}} \int
\eta a(x,\rho,\xi,\eta)
\widehat{f}(\xi,\eta) e^{ix\xi} e^{i\rho\eta}  d\xi d\eta\\
=& \frac{i}{(2\pi)^{n+1}} \int
(\eta-q_1(x,\rho,\xi)) a(x,\rho,\xi,\eta)
\widehat{f}(\xi,\eta) e^{ix\xi} e^{i\rho\eta}  d\xi d\eta\\
& \quad + \frac{i}{(2\pi)^{n+1}} \int
a(x,\rho,\xi,\eta)
q_1(x,\rho,\xi) \widehat{f}(\xi,\eta) e^{ix\xi} e^{i\rho\eta}  d\xi d\eta.
\end{align*} 
Both terms on the right-hand side involve operators of 
 order $1-|k|$ and so the induction hypothesis
applies to handle higher order derivatives, and the theorem is
 proved for integer $s\ge |k|$.  The general case follows by Sobolev interpolation.
}
\end{proof}

In the case 
$A\in \Psi^k(\mathbb{R}^{n+1})$
for $k\le -1$ without additional 
assumptions on the symbol,
  $\sigma(A)(x,\rho,\xi,\eta)$,
which for instance arises from error terms in a
 symbol expansion, we can still derive 
estimates, up to certain order.
\begin{thrm}
\label{estErrorEllBndry}
Let $g\in {\mathscr{D}'}(\mathbb{R}^{n+1})$ of
	the
	form $g(x,\rho) =g_b(x) \delta(\rho)$ for
$g_b\in W^{-1/2}(R^{n})$.  Let
$A\in \Psi^{k}(\mathbb{R}^{n+1})$, $k\le -1$.  Then for 
 integer $s\le |k|-1$, 
\begin{equation*}
	\|\varphi A g\|_{W^s(\mathbb{H}^{n+1}_-)} \lesssim \|
	g_b\|_{W^{s+k+1/2}(\mathbb{R}^{n})}
\end{equation*}
for any $\varphi\in C^{\infty}_0
(\overline{\mathbb{H}}^{n+1}_-)$.
\end{thrm}
\begin{proof}
	As in the proof of Theorem \ref{estinvell},
we write
\begin{equation*}
A g=
\frac{1}{(2\pi)^{n+1}} \int
a(x,\rho,\xi,\eta)
\widetilde{g}_b(\xi) e^{ix\xi}
e^{i\rho \eta}  d\xi d\eta,
\end{equation*}
and estimate
\begin{align*}
 \left\| \varphi \partial_x^{\alpha} 
\partial_{\rho}^{\beta} A g
 \right\|_{L^2}^2 \lesssim&
\int \frac{\xi^{2|\alpha|}
 \eta^{2\beta}}
 {(1+\eta^2+\xi^2)^{|k|}} 
|\widetilde{g}_b(\xi)|^2 d\eta d\xi\\
\lesssim &
 \int \frac{\xi^{2(|\alpha|+\beta)} }
 {(1+|\xi|)^{2|k|-1}} 
 |\widetilde{g}_b(\xi)|^2  d\xi\\
\lesssim& 
 \| g_b\|_{W^{s+k+1/2}(\mathbb{R}^{n})}^2.
\end{align*}
 This handles the case of $s\ge 0$.
Negative values of $s$ can be handled by
 writing
$\|\varphi A g\|_{W^s(\mathbb{H}^{n+1}_-)} 
 \simeq 
\| \phi \Lambda^{-|s|} \circ   \varphi A g\|_{L^2(\mathbb{H}^{n+1}_-)} $
where
$\phi \in C^{\infty}_0(\overline{\mathbb{H}}^{n+1}_-)$ 
is such that
 $\phi =1 $ on $\mbox{supp}(\varphi)$
 and $\Lambda^{-|s|}$ is a pseudodifferential operator
 with symbol
\begin{equation*}
\sigma\left( \Lambda^{-|s|}\right)
 = \frac{1}{(1+\xi^2+\eta^2)^{|s|/2}},
\end{equation*}
 and then applying the theorem to 
$\Lambda^{-|s|} \circ \varphi A$.
\end{proof}

In the case of a distribution
 supported on the half-space, we have
the following Theorem (see also 
Proposition 3.8 in \cite{BdM71}
 or, for the analogue in the
 case of global regularity,
 Theorem 5.2.5 in \cite{CP}).
\begin{thrm}
	\label{estErrorEllInt}
	Let $f\in L^2(\mathbb{H}^{n+1}_-)$.
  Let $A\in \Psi^{k}(\mathbb{R}^{n+1})$, $k\le -1$.
   Then for $0\le s\le |k|$
\begin{equation*}
	\|\varphi A f\|_{W^s(\mathbb{H}^{n+1}_-)} \lesssim 
	\|f\|_{L^2(\mathbb{H}^{n+1}_-)}
\end{equation*}
for any $\varphi\in C^{\infty}_0
(\overline{\mathbb{H}}^{n+1}_-)$.
\end{thrm}
\begin{proof}
We write
\begin{equation*}
A f=
\frac{1}{(2\pi)^{n+1}} \int
a(x,\rho,\xi,\eta)
\widehat{f}(\xi,\eta) e^{ix\xi}
e^{i\rho \eta}  d\xi d\eta,
\end{equation*}
and estimate
\begin{align*}
\left\| \varphi \partial_x^{\alpha} 
\partial_{\rho}^{\beta} A g
\right\|_{L^2}^2 \lesssim&
\int \frac{\xi^{2|\alpha|}
	\eta^{2\beta}}
{(1+\eta^2+\xi^2)^{|k|}} 
|\widehat{f}(\xi,\eta)|^2 d\eta d\xi\\
\lesssim& 
\|f\|_{L^2(\mathbb{H}^{n+1}_-)}^2.
\end{align*}
\end{proof}

We can combine 
 Theorems
 \ref{estinvell} and \ref{estErrorEllBndry}
 (respectively Theorems \ref{thrmSobIntEst} and
  \ref{estErrorEllInt})
and apply them to operators which can be decomposed
 into an operator satisfying the hypothesis
  of Theorem \ref{estinvell} 
  and a remainder term.

\begin{defn}
\label{defnDecomp}
	We say an operator $B\in \Psi^{-k}(\mathbb{R}^{n+1})$
	for $k\ge 1$ is
	{\it decomposable} if for any $N\ge k$ 
	it can be written in the form 
	\begin{equation*} 
	B = A + A_{-N},
	\end{equation*}
	where $A\in \Psi^{-k}(\mathbb{R}^{n+1})$ is an operator 
	satisfying the hypothesis
	of Theorem \ref{estinvell} .
\end{defn}

We recall the discussion in the 
 Introduction and remark here that for our purposes we 
 could have used the definition of operators with
the transmission property as defined by
 Boutet de Monvel \cite{BdM71}
   as in the applications which are
to appear, all our decomposable 
 operators can be reduced to 
 or can be replaced with such operators
with the transmission property (with perhaps some 
 trivial modifications for the order of the operator).  
This would require however a significant
 amount of explanation and some not so enlightening
calculations to show how they can be reduced to the
 case of Boutet de Monvel.

 As we mentioned in the Introduction, on a smoothly bounded domain,
$\Omega$, we can localize by considering a covering of 
 $\Omega$ so that in each set of the covering there exist
local coordinates, $(x,\rho)$, and then apply the above analysis
 on half-spaces to the domain with (smooth) boundary, $\Omega$.
  We can then define $\Psi^k(\Omega)$ by using local
coordinate charts and defining $\Psi^k(\mathbb{H}_{-}^{n+1})$ as
 the restriction of an operator in $\Psi^k(\mathbb{R}^{n+1})$
  to $\rho<0$.
  
We use the notation $\Psi_b^k(\mathbb{R}^n)$,
respectively $\Psi_b^k(\partial\Omega)$
 in the case of pseudodifferential operators
  on a domain $\Omega\subset \mathbb{R}^{n+1}$,
to denote the space of pseudodifferential
operators of order $k$ on $\mathbb{R}^{n}=\partial
\mathbb{H}^{n+1}_-$, 
respectively $\partial\Omega$.  Further
following our use of the notation $A_k$ to denote any operator
belonging to the family $\Psi^k(\mathbb{H}^{n+1}_-)$ (respectively
$\Psi^{k}(\Omega)$) when acting on distributions
$\phi\in \mathscr{E}'(\mathbb{H}^{n+1}_-)$ (respectively in $\mathscr{E}'(\Omega)$)
we write for 
 $\phi_b\in\mathscr{E}'(\mathbb{R}^{n})$ 
(respectively in $\mathscr{E}'(\partial\Omega)$)
$A_{k,b}\phi_b$, $A_{k,b}$ denoting a
pseudodifferential operator of order $k$ on the appropriate boundary of a domain.

With coordinates $(x_1,\ldots,x_n,\rho)$ in $\mathbb{R}^{n+1}$,
let $R$ denote the restriction operator, $R:
{\mathscr{D}'}(\mathbb{R}^{n+1}) \rightarrow {\mathscr{D}'}(R^{n})$,
given by
$R\phi=\left.\phi\right|_{\rho=0}$.

\begin{lemma}
\label{restrict} 
Let $g\in {\mathscr{D}'}(\mathbb{R}^{n+1})$ of the
	form $g(x,\rho) =g_b(x) \delta(\rho)$ for
	$g_b\in {\mathscr{D}'}(R^{n})$.  Let
	$A\in \Psi^k(\mathbb{R}^{n+1}) $, be an operator of order $k$,
	for $k\le -2$.  
Then $R \circ A$ induces a pseudodifferential
 operator in $\Psi^{k+1}_b(\mathbb{R}^n)$ 
acting on $g_b$ via
\begin{equation*}
 R\circ A g = A_{k+1,b} g_b.
\end{equation*}
\end{lemma}
\begin{proof}
Denote the symbol of $A$ with
 $a(x,\rho,\xi,\eta)$.
	The symbol
\begin{equation*}
	\alpha(x,\rho,\xi)
	= \frac{1}{2\pi} \int_{-\infty}^{\infty}
	a(x,\rho,\xi,\eta) d\eta
\end{equation*}
(for any fixed $\rho$) belongs to the class
$\mathcal{S}^{k+1}(\mathbb{R}^{n})$, 
 which follows from the
	properties of $a(x,\rho,\xi,\eta)$ as a member of
	$\mathcal{S}^{k}(\mathbb{R}^{n+1})$ and differentiating under the
	integral.
The composition $R \circ A g$ is given by
\begin{align*}
	\frac{1}{(2\pi)^{n+1}} \int a(x,0,\xi,\eta) \widetilde{g}_b(\xi)&
	e^{i{\mathbf x}\cdot \xi} d\xi d\eta\\
 =&
	\frac{1}{(2\pi)^{n}} \int \left[ \frac{1}{2\pi}
	\int_{-\infty}^{\infty} a(x,0,\xi,\eta) d\eta\right]  \widetilde{g}_b(\xi)
	e^{i{\mathbf x}\cdot \xi} d\xi\\
	=&
	\frac{1}{(2\pi)^{n}} \int \alpha(x,0,\xi)  \widetilde{g}_b(\xi)
	e^{i{\mathbf x}\cdot \xi} d\xi\\
	=& A_{k+1,b} g_b.
\end{align*}
\end{proof}
\begin{remark} We can generalize Lemma \ref{restrict} to decomposable operators
	of order $k=-1$ by using the
	residue calculus to integrate out the $\eta$ variable.  See also
Theorem 5.2.4 $ii)$ of \cite{CP}.
\end{remark}

We work directly
 with inverses to elliptic operators and as such
we consider symbols which are also dependent on 
 the $\rho$ variable.  Even if we were to reduce our
operators to the case handled by the transmission
 property, we would need a way to deal with the
$\rho$ dependence.  The following lemma is useful
 in illustrating the effect of multiplication by 
a factor of $\rho$ with an operator (while operating
 on a boundary distribution).  
\begin{lemma}
\label{liglem}
Let $g\in {\mathscr{D}'}(\mathbb{R}^{n+1})$ of the
form $g(x,\rho) =g_b(x) \delta(\rho)$ for
$g_b\in {\mathscr{D}'}(R^{n})$ .  Let
$A\in \Psi^k(\mathbb{R}^{n+1})$, be a
 pseudodifferential operator of order $k$.
Let $\rho$ denote the operator of
 multiplication with $\rho$. Then
$\rho\circ A$ induces a pseudodifferential 
operator of order $k-1$
	on $g$:
\begin{equation*}
	\rho A g = A_{k-1} g.
\end{equation*}
\end{lemma}
\begin{proof}
We write the symbol of the operator $A$ symbol
 as
$a(x,\rho,\xi,\eta)$: $A=Op(a)$.
Since $a(x,\rho,\xi,\eta)$ is of order
$k$, 
$\rho \cdot a(x,\rho,\xi,\eta)$ 
is also of order $k$, and
\begin{align*}
	\rho\circ A (g)=&
	\int \rho a(x,\rho,\xi,\eta)  \widetilde{g}_b(\xi)
	e^{ix\xi} e^{i\rho\eta} d\xi d\eta
	\\
	=& -i \int a(x,\rho,\xi,\eta)
	\widetilde{g}_b(\xi)
	e^{ix\xi} \frac{\partial}{\partial\eta}
	e^{i\rho\eta} d\xi d\eta\\
	=& i \int \frac{\partial}{\partial\eta}
	\Big(a(x,\rho,\xi,\eta)\Big)
	\widetilde{g}_b(\xi)
	e^{ix\xi}
	e^{i\rho\eta} d\xi d\eta\\
	=& A_{k-1} g,
	\end{align*}
	as $\frac{\partial}{\partial\eta}
	\Big(a(x,\rho,\xi,\eta)\Big)$ is a symbol of class $\mathcal{S}^{k-1}(\mathbb{R}^{n+1})$.
\end{proof}

Lemma \ref{restrict} concerned itself with the restrictions of
pseudodifferential operators (applied to distributions supported
on the boundary) to the boundary, while Theorem \ref{estinvell}
allows us to consider pseudodifferential operators applied to
restrictions of distributions.  A special case of 
 Theorem \ref{estinvell} is
\begin{lemma}
\label{ellrest}
	Let
	$A\in \Psi^k(\mathbb{R}^{n+1})$, for $k\le -1$, be 
a decomposable operator.
	Then
\begin{equation*}
A \circ  R \circ A_{-\infty} :
\mathscr{E}'(\mathbb{H}_-^{n+1} )
  \rightarrow C^{\infty}(\mathbb{H}_-^{n+1})
\end{equation*}
i.e., $A \circ  R \circ
 A_{-\infty}=A_{-\infty}$.
\end{lemma}
\begin{proof}
	Let $f\in \mathscr{E}'(\mathbb{H}_-^{n+1} )$ and
	apply Theorem \ref{estinvell} 
(for decomposable operators)
 with $g_b = R\circ A_{-\infty} f$.
	Then for all $s$
\begin{align*}
\|A \circ R \circ A_{-\infty}
	 f\|_{W^s(\mathbb{H}^{n+1}_-)} \lesssim&
	\| R\circ A_{-\infty} f \|_{W^{s+k+1/2}(\mathbb{R}^{n})}\\
	\lesssim &
	\|A_{-\infty} f \|_{W^{\max(1,s+1)}(\mathbb{R}^{n+1})}\\
	\lesssim &
	\| f \|_{W^{-\infty}(\mathbb{H}_-^{n+1})}.
\end{align*}
	The lemma thus follows from the Sobolev Embedding Theorem.
\end{proof}

Similarly proven is the
\begin{lemma}
	\label{ellrestbnd}
	Let
	$A\in \Psi^k(\mathbb{R}^{n+1})$, for $k\le -1$,
 be a decomposable operator.
	Then
\begin{equation*}
	A \circ   A_{-\infty,b} :
	\mathscr{E}'(\mathbb{R}^{n} ) \rightarrow
	C^{\infty}(\mathbb{R}^{n+1}).
\end{equation*}
\end{lemma}

\section{Analysis on intersections of half-spaces}
\label{secInterHalf}

Another situation in which the above
 analysis of pseudodifferential operators
can be applied, {and the situation which is
 our main interest in writing this
  article}, is on an (non-degenerate)
 intersection of smooth domains.  Localizing the
problem in analogy to the localizations of
 Section \ref{secPseudohalf} allows us to
 represent each
domain composing the intersection as a separate
 half-space.  With appropriate choice of metric
the domain can be modeled by the intersection 
 of several half-spaces.  In this section we 
study some properties of pseudodifferential 
 operators on such spaces.  The 
 motivation for this study is the application of
the following results to the study of elliptic operators
 on intersection domains,
 and in particular to be able to obtain weighted
 estimates for solutions to elliptic
problems on the intersection of smooth domains.

 In this section, $\rho$ 
  is a variable in $\mathbb{R}^m$
   for $m\le n$: $\rho=(\rho_1,\ldots\rho_m)$,
 and $x = (x_1,\ldots, x_{n-m})$
  is an $n-m$ dimensional variable.
We define the half-spaces
\begin{equation*}
\mathbb{H}_j^n =
\{(x,\rho) \in \mathbb{R}^n
: \rho_j<0 \}.
\end{equation*}
With a multi-index
$I=(i_1,\ldots, i_k)$, we denote the intersection of
half-spaces
\begin{equation*}
\mathbb{H}_I^n =
\bigcap_{j\in I} \mathbb{H}_j^n.
\end{equation*}
 The convention used here is
$\mathbb{H}_I^n = \mathbb{R}^n$ when
 $I=\emptyset$.
For this section, we will fix $I$
with $|I|=m$.
 Without loss of generality,
$I=(1,\ldots,m)$.

We use the multi-index notation:
\begin{equation*}
\rho_{J}^{\alpha}
= \prod_{j\in J} \rho_j^{\alpha_j},
\end{equation*}
for $\alpha = 
 (\alpha_1,\ldots, \alpha_{|J|})$ a multi-index. 
 To indicate a missing index, $j$, we use the
notation $\hat{j}$.  Thus we write
\begin{equation*}
 I_{\hat{j}} : = I \setminus \{j\}.
\end{equation*}
For ease of notation, in place of
$\rho_{I_{\hat{j}}}^{\alpha}$,
we write
\begin{equation*}
\rho_{\hat{j}}^{\alpha}
= \rho_1^{\alpha_1} \cdots \rho_{j-1}^{\alpha_{j-1}}
\rho_{j+1}^{\alpha_{j+1}} \cdots
\rho_m^{\alpha_{m}}.
\end{equation*}
Similarly, we write
\begin{equation*}
\rho_{\hat{k}\hat{j}}^{\alpha} = 
\prod_{{i\in I}\atop
	{i\neq j,k}} \rho_i^{\alpha_i}.
\end{equation*}
In the case we have equal powers,
\begin{equation*}
 \alpha_1 = \cdots =\alpha_{j-1}=
  \alpha_{j+1} = \cdots = \alpha_m
 = r,
\end{equation*}
we write
\begin{equation*}
 \rho_{\hat{j}}^{r\times(m-1)}
   := \rho_{\hat{j}}^{\alpha}.
\end{equation*}

We now define the weighted Sobolev norms
on the half-spaces, for $\alpha\in \mathbb{R}$,
 and
$s,k \in\mathbb{N}$:
\begin{multline*}
W^{\alpha,s} (\mathbb{H}_I^n,
 \rho,k)
= \\
 \left\{
f\in W^{\alpha}(\mathbb{H}_I^n) 
\Big| \rho^{(sk - rk)\times m }f \in
W^{\alpha + s-r}(\mathbb{H}_I^n) 
\mbox{ for each } 0\le r\le s
\right\} 
\end{multline*}
with norm
\begin{equation*}
\left\| f \right\|_{
	W^{\alpha,s} (\mathbb{H}_I^n,\rho,k) 
} 
= \sum_{0\le r \le s} 
\left\|\rho^{(sk - rk)\times m} f \right\|_{
	W^{\alpha + s -r} (\mathbb{H}_I^n) 
}. 
\end{equation*}
Similar weighted spaces can be defined 
 with one (or more) $\rho_j$ terms missing.
For example,
\begin{multline*}
W^{\alpha,s} (\mathbb{H}_I^n,
\rho_{\hat{j}},k)
= \\
\left\{
f\in W^{\alpha}(\mathbb{H}_I^n) 
\Big| \rho_{\hat{j}}^{(sk - rk)\times (m-1) }f \in
W^{\alpha + s-r}(\mathbb{H}_I^n) 
\mbox{ for each } 0\le r\le s
\right\} 
\end{multline*}
with norm
\begin{equation*}
\left\| f \right\|_{
	W^{\alpha,s} (\mathbb{H}_I^n,\rho_{\hat{j}},k) 
} 
= \sum_{0\le r \le s} 
\left\|\rho_{\hat{j}}^{(sk - rk)\times (m-1)} f \right\|_{
	W^{\alpha + s -r} (\mathbb{H}_I^n) 
}. 
\end{equation*}

In the case $k=1$ we shall use the notation,

\begin{equation*}
W^{\alpha,s}
\left(\mathbb{H}_I^n,\rho
\right)
:= W^{\alpha,s} (\mathbb{H}_I^n,
\rho,1)
\end{equation*}
 which has the norm
\begin{equation*}
\| f \|_{W^{\alpha,s}
	\left(\mathbb{H}_I^n,\rho
	\right)}
 = \sum_{r=0}^s \| \rho^{r\times m}
 f \|_{ W^{\alpha+r}
 	\left(\mathbb{H}_I^n\right)}.
\end{equation*}

As there are several
 domains whose boundaries
make up the boundary of
 an intersection domain, we
  use a subscript 
to indicate a pseudodifferential
 operator in the class of operators on a 
  specific boundary.  
For example, if $A\in \Psi^{\alpha}(\partial
 \mathbb{H}^n_j)$, we write
$A = A_{\alpha,bj}$.  We adhere to 
 the convention that 
$A_k$ denotes an operator in 
 $\Psi^{k}(\mathbb{R}^n)$, and
that $A_{k,bj }$ denotes an operator in 
$\Psi^{k}(\partial \mathbb{H}^n_j)$.

To denote extensions by 0 across
$\rho_j=0$ to $\rho_j>0$,
we use the
superscript $E_j$:
let $g\in L^2(\mathbb{H}_I^n)$; then
$g^{E_j}\in L^2 
\left(\mathbb{H}_{I_{\hat{j}}}^n
\right)$ is 
defined by
\begin{equation*}
g^{E_j} = 
\begin{cases} g &\mbox{
	if } \rho_j < 0  \\ 
0 & \mbox{if } \rho_j \ge  0.
\end{cases}
\end{equation*} 
Similarly, for a multi-index, $J$, we
define
$ g^{E_J} \in L^2\left(
\mathbb{H}_{I\setminus J}^n \right)$ by
$ g^{E_J} = g$ on 
$\mathbb{H}_I^n$ and 0 elsewhere
(that is for any 
$(x,\rho)\in \mathbb{H}_{I\setminus J}^n$
 for which any $\rho_j\ge 0$ for $j\in J$). 

{
One of the (equivalent) definitions of 
 Sobolev spaces on Lipschitz domains 
(which applies to our case of intersection domains)
 relies on first defining the Sobolev 
  spaces in $\mathbb{R}^n$ and then 
 restricting functions defined in all of
$\mathbb{R}^n$ to a bounded Lipschitz domain.  The next
 Lemmas show that multiplication by factors of the 
  defining functions allows one to consider extensions by
 zero as the functions on which to apply restrictions.  
}

We establish
\begin{lemma}
\label{rhoOddReflect}
	Let $g\in W^s
	\left(\mathbb{H}_I^n\right)$
for some integer $s\ge 0$.
	Then $\rho_j^s g^{E_j}\in
	W^s\left(\mathbb{H}_{I_{\hat{j}}}^n
	\right)$.
\end{lemma}
\begin{proof}
	We only need to check the derivatives 
	with respect to $\rho_j$.  We have
\begin{equation}
\label{derWithDelta}
	\partial_{\rho_j}^s \left(
	\rho_j^s g^{E_j}
	\right) = \sum c_k \rho_j^{s-k} 
	\partial_{\rho_j}^{s-k} 
	g^{E_j}.
\end{equation}
$\partial_{\rho_j}^{s-k} g^{E_j}$
	itself is a sum of terms
	of (derivatives of) 
	delta functions,
$\delta^{(i)}$, $i\le s-k-1$,
 in addition to the extension of terms
   $\left(\partial_{\rho_j}^{s-k-i-1} 
	g^{E_j}\right)\Big|_{\mathbb{H}_I^n}$:
\begin{equation*}
 \partial_{\rho_j}^{s-k} g^{E_j}
	= \sum_{i=0}^{s-k} d_i
	\delta^{(i-1)}(\rho_j)
	\left(\partial_{\rho_j}^{s-k-i} 
	g\right)^{E_j},
\end{equation*}
	where we consider $\delta^{-1}\equiv 1$, 
	{and the $d_i$ are constants with $d_0 = 1$.}
	Inserting this into \eqref{derWithDelta},
	the delta functions combine with the
	powers of $\rho_j$ to yield zero, and we have
	\begin{equation*}
	\partial_{\rho_j}^s \left(
	\rho_j^s g^{E_j}
	\right) = \sum c_k \rho_j^{s-k} 
	\left(\partial_{\rho_j}^{s-k} 
	g\right)^{E_j}.
	\end{equation*}
	The lemma now follows 
(in the case $s$ is an integer)
 by the assumption
	on the regularity of $g$
	in $\mathbb{H}_I$.
\end{proof}

A similar proof shows
\begin{lemma}
\label{oddReflectWithRho}
	Let $s\ge 0$ be an integer,
	$k \in \mathbb{N}$, and 
	$\rho_j^{rk} g\in W^{r-\alpha}
	\left(\mathbb{H}_I^n\right)$
	for integers $r\le s$ and
$\alpha\ge 0$.
	Then $\rho_j^{sk} g^{E_j}\in
	W^{s-\alpha}
\left(\mathbb{H}_{I_{\hat{j}}}^n
	\right)$.
\end{lemma}

For a mutli-index,
$J$, let us denote
\begin{equation*}
\mathbb{H}_{J,bk}^{n-1}
:= \partial\mathbb{H}_k^n
\bigcap \mathbb{H}_{J_{\hat{k}}}^{n},
\end{equation*}
with the convention
 $J_{\hat{k}} = J$ in the case 
$k\notin J$.

 We present the following Theorem which is
  a weighted analogue of Theorem
 \ref{estinvell}.
In the following Theorem we use the notation 
 $\delta_j:= \delta(\rho_j)$.  
 {
 A pseudodifferential 
  operator will be applied to a distribution supported on the
  boundary $\mathbb{H}_{I,bj}^{n-1}$ and to apply Theorem 
\ref{estinvell} we look at the hypotheses with respect to $\eta_j$,
 the dual
 variable to $\rho_j$; for instance, the symbol of the operator will be meromorphic 
with respect to $\eta_j$ with poles giving symbols of order 1 operators.
 We say in this case that the operator satisfies the hypotheses of Theorem
  \ref{estinvell} with respect to $\mathbb{H}_{I,bj}^{n-1}$.
}
\begin{thrm}
 	\label{ellBndryDist}
 	Let $A$ be a pseudodifferential operator
 	(of order $-\alpha\le -1$)
 	satisfying the hypotheses of Theorem
 	\ref{estinvell} with respect to $\mathbb{H}_{I,bj}^{n-1}$.
 	Let $0 \le \gamma <1/2$ and $g_b\in W^{\gamma,s}
 \left(
 \mathbb{H}_{I,bj}^{n-1} ,\rho_{\hat{j}},k
 \right)$ with compact support in 
$\overline{\mathbb{H}}_{I,bj}^{n-1}$.
Then, for $\beta - 1/2$ a non-negative integer with
  $ \beta -\alpha \le\gamma$,
\begin{equation*}	
 	\rho^{rk \times m}
 	A\left(g_b^{E_{I_{\hat{j}}}}
 	\times \delta_j\right)
 	\in
 	W^{r+\beta-{1/2}}\left(\mathbb{R}^n\right)
\end{equation*}
 	for all $r\le s$,
 	and
\begin{equation*}
 	\left\|
 	A\left(g_b^{E_{I_{\hat{j}}}}
 	\times \delta_j\right)
 	\right\|_{W^{\beta-1/2,s}
 		\left(\mathbb{H}^n_I,
 		\rho,k\right)}
 	\lesssim 
 	\| g_b\|_{W^{\beta-\alpha,s}
 		\left(
 		\mathbb{H}_{I,bj}^{n-1}
 		,\rho_{\hat{j}}	,k\right)}.
\end{equation*}

The estimates also hold for all 
 $\beta\ge 1/2$ with the property
$-1/2\le \beta-\alpha \le \gamma$.
 \end{thrm}
 \begin{proof}
 	We prove the Theorem in the
 	case $\beta=\alpha+\gamma$.  
 	The general case follows
 	the same steps.	 
 	
We use an operator $\Lambda_{bj}$ 
 on $\partial \mathbb{H}_j^{n}$, which is defined 
in analogy to the operator $\Lambda$:
\begin{equation*}
\sigma\left( \Lambda_{bj}^{k}\right)
= 
(1+\xi^2+\eta_{\hat{j}}^2)^{k/2}.
\end{equation*}
With our notation, $\eta_{\hat{j}}$ is
 understood to denote
$(\eta_1, \ldots, \eta_{j-1}, \eta_{j+1}, 
 \ldots, \eta_m)$ and
\begin{equation*}
 \eta_{\hat{j}}^2
  = \eta_1^2 + \ldots + \eta_{j-1}^2+ \eta_{j+1}^2 
    + \ldots \eta_m^2.
\end{equation*}

{ 	We recall that for a bounded 
Lipschitz domain, $\Omega$, 
the operator defined by extension by zero
outside of $\Omega$ to all of $\mathbb{R}^n$
 is bounded on $W^{\gamma}(\Omega)$
(see Theorem 3.33 in \cite{McLean}).}
Thus, we have
 	\begin{equation*}
 	g^{E_{I_{\hat{j}}}}_b \in  W^{\gamma}
 	\left(\partial\mathbb{H}_j^{n}\right),
 	\end{equation*}  
 	or 
 	\begin{equation*}
 	\Lambda^{\gamma}_{bj} g^{E_{I_{\hat{j}}}}_b \in  
 	L^{2}\left(\partial\mathbb{H}_j^{n}\right).
 	\end{equation*} 
 By assumption
\begin{equation}
 \label{lambdaHalf}
 D^{r}_{bj}	\rho_{\hat{j}}^{rk\times(m-1)} 
 	 g_b \in  
 	W^{\gamma}
 	\left(
 	\mathbb{H}_{I,bj}^{n-1}	\right)
\end{equation}
 	for $0\le r \le s$, and 
$D^{r}_{bj}$ a differential operator of 
 order $r$ on $\partial\mathbb{H}_j^{n}$.
Using extensions by zero, we see
\begin{equation*}
	D^{r}_{bj}	\rho_{\hat{j}}^{rk\times(m-1)} 
	g_b^{E_{I_{\hat{j}}}} \in  
	W^{\gamma}
	\left(
	\partial\mathbb{H}_j^{n}	\right)
\end{equation*}
and thus
\begin{equation*}
 	\rho_{\hat{j}}^{rk\times(m-1)}
 	 g_b^{E_{\hat{j}}} \in  
 	W^{r+\gamma}
 	\left(
 	\partial \mathbb{H}_{j}^{n}	\right).
\end{equation*}

Write $g_j = g_b^{E_{{I}_{\hat{j}}}}\times \delta_j$.
We use Lemma \ref{liglem} to write
\begin{equation*}
 	\rho_j^{rk}
 	A g_j =
 	A_{-rk-\alpha} g_j.
\end{equation*}
 	We have
\begin{align*}
 \rho^{rk\times m} 
 	A g_j
 	=& \rho_{\hat{j}}^{rk\times (m-1)}
 	\rho_j^{rk}
 	A_{-\alpha} g_j \\
 	=& \rho_{\hat{j}}^{rk\times (m-1)}
 	A_{-rk-\alpha}g_j
 	\\
 	=&\sum_{l=0}^{r}
 	  A_{-\alpha-rk-(r-l)} 
 	\left(
 		\rho_{\hat{j}}^{lk\times(m-1)} g_j
 	\right).
 \end{align*}

 The $A_{-\alpha-rk-(r-l)} $ operators
 above satisfy the
 	hypotheses of Theorem \ref{estinvell},
 	while
\begin{equation*}
 \rho_{\hat{j}}^{lk\times(m-1)}
   g_b^{E_{{I}_{\hat{j}}}} \in
 	W^{l+\gamma}(\partial\mathbb{H}_j^n).
\end{equation*} 	
 	Therefore,
 	Theorem \ref{estinvell} applies to give
 	the estimates
\begin{align*}
 	\nonumber
 	\sum_l \Big\|  
A_{-\alpha-rk-(r-l)}
 \left(
 \rho_{\hat{j}}^{lk\times(m-1)} g_j
 \right)& 
 	\Big\|_{W^{r+\alpha+\gamma - 1/2}
 		\left(
 		\mathbb{R}^n
 		\right)}\\
\lesssim&
 	\sum_l \Big\|  
A_{-\alpha-rk-(r-l)}
\left(
\rho_{\hat{j}}^{lk\times(m-1)} g_j
\right)
\Big\|_{W^{rk+r+\alpha+\gamma - 1/2}
	\left(
	\mathbb{R}^n
	\right)}
 \\
\lesssim&
 \sum_l \left\| 
     \rho_{\hat{j}}^{lk\times(m-1)}
 	g_b
 	\right\|_{W^{\gamma
 		 +l}
 		\left( \mathbb{H}_{I,bj}^{n-1}\right)}\\
 	\lesssim& 
 	\left\| 
 	g_b
 	\right\|_{W^{\gamma,r}
 		\left(
 		\mathbb{H}_{I,bj}^{n-1}
 		,\rho_{\hat{j}},k	\right)}
 	.
 	\end{align*}
 \end{proof}

The assumption that $g_b$ has compact support in
 $\overline{\mathbb{H}}_{I,bj}^{n-1}$ was used in order to apply
extensions by zero.  Such an assumption will not be needed
 when the analysis on bounded domains is applied. 

For operators which do not satisfy the 
 hypotheses of Theorem
\ref{estinvell} we can derive estimates in a similar
 manner to the method of 
Theorem \ref{estErrorEllBndry}.
\begin{thrm}
\label{estErrorWeighted}
  Let
	$A \in \Psi^{-\alpha}
	(\mathbb{R}^{n})$, $-\alpha \le -1$.
Then for $\beta-1/2$ an integer with
 $\beta\le\alpha-1/2$,
and $g_b\in W^{ 0,s}
\left(
\mathbb{H}_{I,bj}^{n-1} ,\rho_{\hat{j}},k
\right)$ with compact support in 
$\overline{\mathbb{H}}_{I,bj}^{n-1}$,
\begin{equation*}
\left\|
A\left(g_b^{E_{I_{\hat{j}}}}
\times \delta_j\right)
\right\|_{W^{\beta-1/2,s}
	\left(\mathbb{H}_I^n,
	\rho,k\right)}
\lesssim 
\| g_b\|_{W^{\beta-\alpha,s}
	\left(
	\mathbb{H}_{I,bj}^{n-1}
	,\rho_{\hat{j}},k	\right)}.
\end{equation*}
\end{thrm}
\begin{proof}
	 The proof of Theorem \ref{ellBndryDist}
applies up to the last estimate where Theorem
 \ref{estErrorEllBndry} is to be applied as opposed to
Theorem \ref{estinvell}.  For this case we need to ensure
 $\beta \le \alpha -1/2$.  
\end{proof}

 In the case of operators acting
 on functions supported on all of
$\mathbb{H}^n_I$ we have the following
 weighted estimates
\begin{thrm}
\label{thrmWghtdFullHI}
 Let $A\in \Psi^{-\alpha}(\mathbb{R}^{n})$
  for $\alpha\ge 1$.  
Let $f\in W^{0,s}
 (\mathbb{H}^n_I,\rho,k)$.  Then
\begin{equation*}
 \| A f\|_{W^{\alpha,s}
 	 (\mathbb{H}_I^n,\rho,k)}
   \lesssim \| f\|_{W^{0,s}
   	(\mathbb{H}^n_I,\rho,k)}.
\end{equation*}
\end{thrm}
\begin{proof}
 We have for $r\le s$,
\begin{align*}
 \rho^{rk\times m}
  A f= \sum_{l=0}^r
  A_{-\alpha-(r-l)}
\left(
	\rho^{lk\times m} f
\right).
\end{align*}
By Lemma \ref{oddReflectWithRho} we have
 $\rho^{lk\times m} f^{E_I} \in 
  W^l(\mathbb{R}^n )$ from which the estimates
\begin{align*}
\left\|
  \rho^{rk\times m} A f
\right\|_{W^{\alpha+r}(\mathbb{R}^n)} 
  \lesssim& 
\sum_{l=0}^r \left\|
	 \rho^{lk\times m} f
\right\|_{W^{l}(\mathbb{H}^n_I)} \\
 \lesssim& 
 \| f\|_{W^{0,r}
	(\mathbb{H}^n_I,\rho,k)}
\end{align*}
follow.  Summing over all $r\le s$ 
 finishes the proof.
\end{proof}

We can improve the above Theorem 
 by removing one of the
$\rho$ components and using Theorem
 \ref{estErrorEllInt}.  
 
{ 
	We say an operator $B\in \Psi^{-k}(\mathbb{R}^{n+1})$
	for $k\ge 1$ is
	{\it decomposable with respect to 
		$\mathbb{H}_{I,bj}^{n-1}$ } if for any $N\ge k$ 
	it can be written in the form 
	\begin{equation*} 
	B = A + A_{-N},
	\end{equation*}
	where $A\in \Psi^{-k}(\mathbb{R}^{n+1})$ is an operator 
	satisfying the hypothesis
	of Theorem \ref{estinvell} with respect to 
	$\mathbb{H}_{I,bj}^{n-1}$.
}

\begin{thrm}
Let $A\in \Psi^{-\alpha}(\mathbb{R}^{n})$
	for $\alpha\ge 1$ be decomposable with respect to 
	$\mathbb{H}_{I,bj}^{n-1}$, for some $j\in I$.
Let $f\in W^{0,s}
	(\mathbb{H}^n_I,\rho_{\hat{j}},k)$.  Then
\begin{equation*}
	\| A f\|_{W^{\alpha,s}
		(\mathbb{H}_I^n,\rho_{\hat{j}},k)}
	\lesssim \| f\|_{W^{0,s}
		(\mathbb{H}^n_I,\rho_{\hat{j}},k)}.
\end{equation*}
\end{thrm}
\begin{proof}
 The proof is almost the
same as that of Theorem \ref{thrmWghtdFullHI}.
We have for $r\le s$,
\begin{align*}
	\rho_{\hat{j}}^{rk\times (m-1)}
	A f= \sum_{l=0}^r
	A_{-\alpha-(r-l)}
	\left(
	\rho_{\hat{j}}^{lk\times (m-1)} f
	\right).
\end{align*}
By Lemma \ref{oddReflectWithRho} we have
	$\rho_{\hat{j}}^{lk\times (m-1)}
	 f^{E_{I_{\hat{j}}}} \in 
	W^l(\mathbb{H}^n_j )$.
Thus, an application of Theorem
\ref{thrmSobIntEst}
 yields 
	\begin{align*}
	\left\|
	\rho_{\hat{j}}^{rk\times (m-1)} A f
	\right\|_{W^{\alpha+r}(\mathbb{R}^n)} 
	\lesssim& 
	\sum_{l=0}^r \left\|
	\rho_{\hat{j}}^{lk\times (m-1)} f
	\right\|_{W^{l}(\mathbb{H}^n_j)} \\
	\lesssim& 
	\| f\|_{W^{0,r}
		(\mathbb{H}^n_I,\rho_{\hat{j}},k)}.
	\end{align*}
Summing over all $r\le s$ 
	finishes the proof.
\end{proof}

When working with boundary value problems
on intersection domains, or intersections of
half-spaces, 
  restrictions to one boundary
 of an operator applied to a distribution with
support on a different boundary arise.  

We let
$R_j$ denote the operator of 
restriction to the boundary, $\rho_j = 0$.
To
 deal with restrictions to one boundary 
  of an operator acting on a distribution supported
 on another boundary, we introduce some 
notation:
 for $\alpha+1/2\in \mathbb{N}$, 
$\alpha\ge 1/2$, and $j\neq k$,
\begin{equation*}
 \lre_{-\alpha}^{jk}:
 W^s\left(\overline{
 	\mathbb{H}}_{I,bj}^{n-1},\rho_{\hat{j}},\lambda \right) 
 \rightarrow
W^{s+\alpha}\left(\overline{
	\mathbb{H}}_{I,bk}^{n-1},\rho_{\hat{k}},\lambda\right)
\end{equation*}
{(with some restriction on $s$ to be introduced)},
where 
 $\lre_{-\alpha}^{jk}$ is of the form
\begin{equation}
\label{formLRE}
\lre_{-\alpha}^{jk} g_b
= R_k \circ B_{-\alpha-1/2} g_j,
\end{equation}
where, as above 
 $g_j := g_b^{E_{I_{\hat{j}}}} \times \delta_j$,
and
 where $B_{-\alpha-1/2}\in \Psi^{-\alpha-1/2}
 \left(\mathbb{R}^{n} \right)$
is decomposable with respect to 
$\mathbb{H}_{I,bj}^{n-1}$.

For some crude estimates
 in the case $1/2\le \beta\le \alpha -1$,
 we could write
\begin{align*}
\left\| \rho_{\hat{k}}^{r\lambda\times(m-1)}
\lre_{-\alpha}^{jk} g_b
\right\|_{W^{r+\beta-1/2}
	\left(\mathbb{H}_{I,bk}^{n-1}\right)}
\lesssim&
\left\| \rho_{\hat{k}}^{r\lambda\times(m-1)}
R_k \circ A_{-\alpha-\frac{1}{2}} g_j
\right\|_{W^{r+\beta-1/2}
	\left(\partial\mathbb{H}_k^n\right)}\\
\lesssim&
\left\| \rho_{\hat{k}}^{r\lambda\times(m-1)}
A_{-\alpha-\frac{1}{2}} g_j 
\right\|_{W^{r+\beta}
	\left(\mathbb{H}_k^n\right)}
\\
\lesssim&
\| g_b\|_{W^{\beta-\alpha,r}
	\left(
	\mathbb{H}_{I,bj}^{n-1}
	,\rho_{\hat{k}\hat{j}},\lambda	\right)},
\end{align*}
where in the last step we use
the estimates from
Theorem \ref{estErrorWeighted}
 (with some restrictions on which sets, such
  as the integers, $\alpha$ and $\beta$ belong to).  
And after summing over
$r\le s$ we would have the estimates
\begin{equation*}
\left\|
\lre_{-\alpha}^{jk} g_b
\right\|_{W^{\beta-1/2,s}
	\left(\mathbb{H}_{I,bk}^{n-1},
	\rho_{\hat{k}}\right)}
\lesssim \| g_b\|_{W^{\beta-\alpha,s}
	\left(
	\mathbb{H}_{I,bj}^{n-1}
	,\rho_{\hat{k}\hat{j}}	\right)}
\end{equation*}
for $\beta\le \alpha -3/2$.

However, we can improve the estimates
 for the operators with meromorphic
symbols in 
two ways.  First, the order of the
Sobolev spaces can be increased by making
use of relations between elliptic
operators acting on distributions supported
on the boundary.  Secondly, there is 
a loss of a factor of $\rho$ in the estimate
above; as $g_b\times \delta_j$ is supported
on $\mathbb{H}_{I,bj}^n$, a weighted estimate, 
using $\rho_{\hat{j}}$ is desired 
 on the right (as 
opposed to $\rho_{\hat{k}\hat{j}}$).
These improvements will be made in the
following Corollary of Theorems
  \ref{ellBndryDist} and \ref{estErrorWeighted}.
\begin{cor}
	\label{rhoLRE}
Let $\lre_{-\alpha}^{jk}$ as above,
 and $g_b\in W^{0,s}
 \left(
 \mathbb{H}_{I,bj}^{n-1} ,\rho_{\hat{j}}, {\lambda}
 \right)$.  Then
for $0\le \beta < \alpha$ 
\begin{equation*}
	\left\|
	\lre_{-\alpha}^{jk} g_b
	\right\|_{W^{\beta,s}
		\left(\partial\mathbb{H}_k^n,
		\rho_{\hat{k}},\lambda\right)}
	\lesssim \| g_b\|_{W^{\beta-\alpha,s}
		\left(
		\mathbb{H}_{I,bj}^{n-1}
		,\rho_{\hat{j}},\lambda	\right)}.
\end{equation*}
\end{cor}
\begin{proof}
For given $\alpha,s$, we choose $N$
 large,  $\alpha  +s \ll N$, and
we write
\begin{equation}
\label{decompCor}
 \lre_{-\alpha}^{jk} g_b
 = R_k \circ A_{-\alpha-\frac{1}{2}} g_j
  + R_k \circ A_{-N} g_j,
\end{equation}
where
 $A_{-\alpha-\frac{1}{2}}$ satisfies the
  hypotheses of Theorem
  \ref{estinvell} with respect to 
  $\mathbb{H}_{I,bj}^{n-1}$.

{ 
We first estimate
$
R_k \circ A_{-N} g_j
$:
\begin{equation*}
\left\|
R_k \circ A_{-N}g_j
\right\|_{W^{\beta+s} 
	(\mathbb{R}^{n-1})}^2 
\lesssim
 	\int 
 \frac{\left(\eta_{\hat{k}}^2 + \xi^2\right)^{\beta+s}}
  {(1+\eta^2
 	+ \xi^2)^{N}}
 |\widetilde{g}_b(\xi,\eta_{\hat{j}})|^2
 d\xi d\eta
\end{equation*} 
Integrating over $\eta_j$ yields
\begin{align*}
\left\|
R_k \circ A_{-N}g_j
\right\|_{W^{\beta+s} 
	(\mathbb{R}^{n-1})}^2 
\lesssim&
\int 
\frac{\left(\eta_{\hat{j}}^2 + \xi^2\right)^{\beta+s}}
{(1+\eta_{\hat{j}}^2
	+ \xi^2)^{N-1/2}}
|\widetilde{g}_b(\xi,\eta_{\hat{j}})|^2
d\xi d\eta \\
\lesssim&
\| g_b
\|^2_{W^{\beta-\alpha}(\mathbb{H}_{I,bj}^{n-1})},
\end{align*} 
from the assumption $\alpha + s \ll N$.

To show
\begin{equation}
\label{estNoRho}
\left\|
R_k \circ A_{-\alpha-\frac{1}{2}}g_j
\right\|_{W^{\beta}
	(\mathbb{R}^{n-1})}^2
\lesssim
\| g_b
\|^2_{W^{\beta-\alpha}(\mathbb{H}_{I,bj}^{n-1})},
\end{equation}
for {$\beta < \alpha $, we estimate}
\begin{align*}
 \left\|
 R_k \circ A_{-\alpha-\frac{1}{2}}g_j
 \right\|_{W^{\beta}
 	(\mathbb{R}^{n-1})}^2
 \lesssim&
 \int 
 \frac{ \left(\eta^2 + \xi^2\right)^{\beta}}
 {(1+\eta^2
 	+ \xi^2)^{\alpha + 1/2}}
 |\widetilde{g}_b(\xi,\eta_{\hat{j}})|^2
 d\xi d\eta\\
  \lesssim&
\int 
\frac{\left(\eta_{\hat{j}}^2 + \xi^2\right)^{\beta}}
{(1+\eta_{\hat{j}}^2
	+ \xi^2)^{\alpha}}
|\widetilde{g}_b(\xi,\eta_{\hat{j}})|^2
d\xi d\eta_{\hat{j}}\\
\lesssim&
\| g_b
\|^2_{W^{\beta-\alpha}(\mathbb{H}_{I,bj}^{n-1})}.
\end{align*}

}

When calculating the weighted estimates 
	on $\mathbb{H}_{I,bk}$,
with weights which do not include
factors of $\rho_k$, we use the identity
\begin{equation*}
{\rho_{\hat{k}}^{s\lambda\times(m-1)}}
R_k \circ A_{-\alpha} g_j
	\simeq R_k\circ
\frac{\partial^s}{\partial\rho_k^s}
	\rho^{s\lambda\times m} A_{-\alpha} g_j,
\end{equation*}
with $\alpha\ge 1$.  Note that the 
	regularity provided by the order,
	$-\alpha\le -1$
 of the pseudodifferential operator ensures
  that derivatives with
respect to $\rho_k$ produce no $\delta$
distributions
(or derivatives of $\delta$) terms. 
Thus, we write 
\begin{equation*}
	\rho_{\hat{k}}^{r\lambda\times(m-1)}
	R_k \circ A_{-\alpha-\frac{1}{2}} g_j
\simeq
R_k \circ
 \frac{\partial^{r\lambda}}
   {\partial\rho_k^{r\lambda}}
	\rho^{r\lambda\times m} 
A_{-\alpha-\frac{1}{2}}g_j,
\end{equation*}
	and further, we use the relation
\begin{align*}
\rho^{r\lambda\times m} 
	A_{-\alpha-\frac{1}{2}} g_j 
=& \rho^{r\lambda\times (m-1)}_{\hat{j}}
	A_{-\alpha-1/2-r\lambda} g_j \\
=&  \sum_{l=0}^{r}
	A_{-\alpha-1/2-r\lambda-(r-l)}
	\left(
	\rho^{l\lambda\times(m-1)}_{\hat{j}} g_j
	\right) 
\end{align*}
	to show
\begin{equation*}
\frac{\partial^{r\lambda}}
 {\partial\rho_k^{r\lambda}}
	\rho^{r\lambda\times m} 
	A_{-\alpha-\frac{1}{2}}g_j
	= \sum_{l=0}^{r}
	A_{-\alpha-1/2-(r-l)}
	\left(
	\rho^{l\lambda\times(m-1)}_{\hat{j}} g_j
	\right).
\end{equation*}
	Therefore, we have
\begin{equation}
 \label{withAlpha}
	\rho_{\hat{k}}^{r\lambda\times(m-1)}
	R_k \circ A_{-\alpha-\frac{1}{2}}
	g_j 
=R_k \circ \sum_{l=0}^{r}
A_{-\alpha-1/2-(r-l)}
\left(
	\rho^{l\lambda\times(m-1)}_{\hat{j}} g_j
\right).
\end{equation}

We note the symbols of the $A_{-\alpha-1/2-(r-l)}$ operators can be bounded by
\begin{equation*}
\frac{(1+\eta_k^2)^{r\lambda/2}}{(1+\eta^2+\xi^2)^{{1/2}(\alpha+1/2+(r-l)+r\lambda)}}
\end{equation*}
modulo lower order symbols with bounds of the same form.  

	The calculation of the estimates then 
{follow those above. 
In particular, for the case 
 $\beta < \alpha$, we have 
\begin{align*}
	\Big\|
\rho_{\hat{k}}^{r\lambda \times(m-1)}
&R_k \circ A_{-\alpha-\frac{1}{2}}
g_j
\Big\|_{W^{\beta+r}
	(\mathbb{R}^{n-1})}\\
& \lesssim \sum_{l=0}^{r}
\int 
\left(\eta^2 + \xi^2\right)^{\beta+r}
\frac{(1+\eta_k^2)^{r\lambda}}{(1+\eta^2+\xi^2)^{\alpha+1/2+(r-l)+r\lambda}}
|\widetilde{h}_j^{l}(\xi,\eta_{\hat{j}})|^2
d\xi d\eta\\
&\lesssim
\sum_{l=0}^{r} \int 
\frac{\left(\eta_{\hat{j}}^2 + \xi^2\right)^{\beta+r+r\lambda}}
{(1+\eta_{\hat{j}}^2
	+ \xi^2)^{\alpha+(r-l)+r\lambda}}
|\widetilde{h}_j^{l}(\xi,\eta_{\hat{j}})|^2
d\xi d\eta_{\hat{j}}\\
&\lesssim
\| g_b
\|^2_{W^{\beta-\alpha,r}(\mathbb{H}_{I,bj}^{n-1})},
\end{align*}
where $h_j^{l} = \rho^{l\lambda\times(m-1)}_{\hat{j}} g_j$.

}	
Summing over $r\le s$ 
give the weighted estimates in terms of 
 $ \| g_b
 \|^2_{W^{\beta-\alpha,s}(\mathbb{H}_{I,bj}^{n-1})},$
as in the statement of the Corollary.   
\end{proof}

 For boundary operators mapping
$\mathbb{H}^n_{I,bj}$ to itself we use the
 notation $\lre_{-\alpha}^{jj}$ to denote
 operators of the form
\begin{equation*}
 \lre_{-\alpha}^{jj} = R_j \circ
B_{-\alpha-1},
\end{equation*}
 where $B_{-\alpha-1} \in 
  \Psi^{-\alpha-1}(\mathbb{R}^n)$ is 
 decomposable with respect to 
 $\mathbb{H}^n_{I,bj}$,
for $\alpha\ge 1$,
 and thus of the form
\begin{equation*}
\lre_{-\alpha}^{jj} =  
 A_{-\alpha, bj},
\end{equation*} 
using Lemma \ref{restrict},
for $\alpha \ge 1$.  We also use the
notation
  to denote
 compositions:
\begin{equation*}
 \lre_{-\alpha}^{jj}
  = \lre_{-\alpha_1}^{kj} \circ
    \lre_{-\alpha_2}^{jk},
\end{equation*}
where $\alpha=\alpha_1+\alpha_2$
 and $\alpha_1,\alpha_2 \ge 1/2$.

As a corollary of  Theorem \ref{thrmWghtdFullHI}, we have
\begin{equation*}
\left\|
	R_j \circ B_{-\alpha-1}g_j
\right\|_{W^{\beta,s}
	\left(\partial\mathbb{H}_j^n,
	\rho_{\hat{j}},\lambda\right)}
\lesssim \| g_b\|_{W^{\beta-\alpha,s}
 	\left(
 	\mathbb{H}_{I,bj}^{n-1}
 	,\rho_{\hat{j}},\lambda	\right)},
\end{equation*}
{for $ \beta < \alpha + 1/2$,}
while Corollary
 \ref{rhoLRE} applied (twice) to 
$\lre_{-\alpha}^{jj}
= \lre_{-\alpha_1}^{kj} \circ
\lre_{-\alpha_2}^{jk}$ yields
\begin{equation*}
\left\|
\lre_{-\alpha_1}^{kj} \circ
\lre_{-\alpha_2}^{jk}g_b
\right\|_{W^{\alpha_1 - \epsilon ,s}
	\left(\partial\mathbb{H}_j^n,
	\rho_{\hat{j}},\lambda\right)}
\lesssim \| g_b\|_{W^{-\alpha_2,s}
	\left(
	\mathbb{H}_{I,bj}^{n-1}
	,\rho_{\hat{j}},\lambda	\right)},
\end{equation*}
{for $\epsilon>0$.}

\section{Example on an intersection of two domains}
\label{secEx}

 We close this paper with a mention of how the
weighted estimates in the previous section may be applied.
The motivation for the introduction of
  the weighted estimates
  is the 
study of estimates for operators of
 boundary value problems such as the 
  Poisson and Green operators.  
Let $\Omega_1, \ldots, \Omega_m \subset
 \mathbb{R}^n$ be smoothly bounded domains
which intersect real transversely.
 That is to say, if $\rho_j$ is a smooth 
defining function for $\Omega_j$, 
 $|d\rho_j| \neq 0$ on $\partial\Omega_j$,
 then
for all $1\le i_1 < \cdots < i_l \le m$, we have
\begin{equation*}
d\rho_{i_1}\wedge \cdots\wedge d\rho_{i_l} \neq 0
\end{equation*}
on $\bigcap_{j=1}^l \partial\Omega_{i_j} 
\cap \partial\Omega$, 
where $\rho_j$ is a defining function for
$\Omega_j$,
We say in this case $\Omega$ is an 
 {\it intersection} domain.  An intersection
  domain is an example of a 
 {\it piecewise smooth} domain
  (see \cite{RS}, from which we base our definition of 
   intersection domains).
Then using a suitable metric 
 locally near a point on 
$\partial\Omega$ the intersection 
 can be modeled by the intersection of
$m$ half-spaces.

 To illustrate some of the reduction to the boundary
techniques on intersection domains, we consider 
 $\Omega = \Omega_1 \cap \Omega_2$, with
  $\Omega_j$ a bounded smooth domain for
$j=1,2$.  We take $\Gamma$ in this section 
 to be a second order elliptic operator,
and consider a Dirichlet problem
\begin{equation*}
\begin{aligned}
\Gamma v =& 0  \qquad \mbox{in } \Omega\\
v =& g \qquad \mbox{on } \partial \Omega.
\end{aligned}
\end{equation*}
The boundary condition can be expressed on
 the individual boundaries as
\begin{equation*}
 v_j = g_j \qquad \mbox{on } \partial \Omega_j,
\end{equation*} 
where $v_j = v|_{\partial\Omega_j}$
 for $j=1,2$
 (with a similar notation holding for
  $g_j$).  
As in the previous Section, we
  will use $R_j$ to denote the operator of 
 restriction to the boundary, $\partial\Omega_j$.
Then with this operator we can also write
 $g_j = R_j v$.
  
Using a partition of unity we assume $v$ and $g$
 have support in a neighborhood of a point of 
 intersection (at which $\rho_1=\rho_2=0$)
  which we take to be the origin.
 We assume $\Gamma$ has a local expression in
a neighborhood containing the support of $v$
 and $g$ of the form
\begin{multline}
\label{pureDiff}
\Gamma = - \partial_{\rho_1}^2
- \partial_{\rho_2}^2
- \sum_{j=1}^{n-2} \partial_{x_j}^2
+ \sum_{ij} c_{ij}(x) \partial_{x_i}
\partial_{x_j}
 + \sum_i b_i(x,\rho) \partial_{x_i}
\\
+ O(\rho_1)A_2 + O(\rho_2)A_2
+s_1(x,\rho) \partial_{\rho_1}
+s_2(x,\rho) \partial_{\rho_2},
\end{multline}
where $\rho_j$ is a defining function 
 for $\Omega_j$ for $j=1,2$, with 
$c_{ij}(x) = O(x)$ and 
 smooth $b_i$ ($1\le i \le n-2$),
$s_1$ and $s_2$.

We let $\eta_j$ be the transform variable dual
 to the $\rho_j$ and $\xi_j$ dual to $x_j$,
and  $\xi = (\xi_1, \ldots, \xi_{n-2})$
 as well as
 $\eta = (\eta_1,\eta_2)$.
In writing \eqref{pureDiff}  as an equation with 
 pseudodifferential operators
we use the following notations with Fourier Transforms.  
The full transform of a function, $w$, will be written
 with the duals to the defining functions first:
\begin{equation*}
 \widehat{w}(\eta,\xi) = \frac{1}{(2\pi)^n}
   \int w(\rho,x) e^{-i\rho\eta} e^{-ix\xi} d\rho dx.
\end{equation*}
Partial Fourier Transforms will be denoted with the 
 $\hat{j}$ notation to indicate which $\rho$ variable 
is not to be transformed.  Thus
\begin{equation*}
F.T._{\hat{1}} w(\rho_1,\eta_2,\xi) = \frac{1}{(2\pi)^{n-1}}
\int w(\rho,x) e^{-i\rho_2\eta_2} e^{-ix\xi} d\rho_2 dx.
\end{equation*}
We also use a non-standard notation to set a specific
 value, which will always be 0, to a variable.
We will use the $\widehat{w}$ notation even to 
 indicate a partial Fourier Transform.  We then rearrange 
 arguments to write that set value first:
\begin{equation*}
 \left. F.T._{\hat{j}} w  \right|_{\rho_j =0} =
\widehat{w}(0_j, \eta_k, \xi),
\end{equation*}
where $k\neq j$.  Thus, for instance,
\begin{equation*}
 \widehat{w}(0_2,\eta_1,\xi) :=
   \frac{1}{(2\pi)^{n-1}} \int  w(\rho_1,0,x) e^{-i\rho_1\eta_1} e^{-ix\xi} d\rho_1 dx.
\end{equation*}
This is not to be confused with
 $\widehat{w}( \eta_1,0, \xi) = \widehat{w}|_{\eta_2=0}$.
 
With this notation, we rewrite \eqref{pureDiff} in terms of Transforms 
  as
\begin{align}
\nonumber
\frac{1}{(2\pi)^{2n}}	\int \bigg(
&\eta^2 +  \xi^2 -
	\sum c_{jk}
\xi_j \xi_k
+\sum_j O(\rho_j) \bigg)  \widehat{v}(\eta,\xi)
e^{i \rho \cdot \eta} e^{ix\xi} d\eta  d\xi\\
\nonumber
&-
\frac{1}{(2\pi)^{2n}}	\sum_{j=1}^2 \int
\bigg(F.T._{\hat{j}} \partial_{\rho_j}
v(0_j,\eta_{\hat{j}},\xi) +
i\eta_j 
\widehat{g}_{j} (\eta_{\hat{j}},\xi) 
\bigg)
e^{i \rho \cdot \eta} e^{ix\xi} d\eta  d\xi\\
\label{origPoissonFourier}
&
+ \sum_{j=1}^{2} S_j
(\partial_{\rho_j} v) +	Bv  = 0,
\end{align}
where $S_j$ is the zeroth operator with 
 symbol $s_j(x,\rho)$, and $B$ a first order
operator with symbol
 $i\sum_{j=1}^{n-2} b_j(x,\rho)\xi_j$.
Note that $S_j
(\partial_{\rho_j} v)$ can be written
\begin{equation*}
 S_j (\partial_{\rho_j} v) = 
  \frac{1}{(2\pi)^{2n}}	\sum_j \int
     s_j(x,\rho) 
   \left(
  \widehat{g}_{j}(\eta_{\hat{j}},\xi) 
+i\eta_j \widehat{v} (\eta,\xi) 
   \right)
    e^{i \rho \cdot \eta} e^{ix\xi} d\eta  d\xi.
\end{equation*}
We define the symbol, $\Xi$, 
 by 
\begin{align*}
\Xi(x,\rho,\xi) = \bigg(
   \eta^2 + \xi^2 -
   \sum c_{jk}
   \xi_j \xi_k
   +\sum_j O(\rho_j) \bigg)^{1/2},
\end{align*}
where the $O(\rho_j)$ terms are those in 
 \eqref{origPoissonFourier}.

Applying an inverse of the highest order terms
 in \eqref{origPoissonFourier}
yields
\begin{align}
\nonumber
v= 
&
\frac{1}{(2\pi)^{2n}}	\sum_{j=1}^2 \int
\frac{F.T._{\hat{j}} \partial_{\rho_j}
v(0_j,\eta_{\hat{j}},\xi) +
i\eta_j 
\widehat{g}_{j} (\eta_{\hat{j}},\xi) }
 {\eta^2+\Xi^2(\rho,x,\xi)}
e^{i \rho \cdot \eta} e^{ix\xi} d\eta  d\xi\\
\nonumber
&
+ A_{-3}(\partial_{\rho}v|_{\partial\Omega})
 + A_{-2} g +A_{-1} v,
\end{align}
locally.

We now expand $\Xi^2(x,\rho,\xi)$ in each
$\rho_j$, writing
$\Xi_{bj} = \Xi|_{\rho_j=0}$:
\begin{align}
\nonumber
v=& 
\sum_{j=1}^2 \frac{1}{(2\pi)^{2n}} \int \frac{
	F.T._{\hat{j}}  \partial_{\rho_j} 
	v(0_j,\eta_{\hat{j}},\xi) +
	i\eta_j \widehat{g}_{j}
	(\eta_{\hat{j}},\xi)}
{\eta^2 +\Xi_{bj}^2}
e^{i \rho \cdot \eta} e^{ix\xi} d\eta  d\xi\\
\nonumber
& + A_{-1} v +
\sum \rho_j
A_{-2}\left(\partial_{\rho_j}
u \big|_{\partial\Omega_j}
\right) +
\sum \rho_j A_{-1} g_{j} \\
\nonumber
&+
A_{-3}\left(\partial_{\rho}
v \big|_{\partial\Omega}
\right) +A_{-2} g.
\end{align}

Using Lemma \ref{liglem} to handle the
$\rho_j$ terms multiplied with 
 the $A_{-1}$ and $A_{-2}$ operators,
we have
\begin{align}
\nonumber
v=& 
\sum_{j=1}^2 \frac{1}{(2\pi)^{2n}} \int \frac{
	F.T._{\hat{j}}  \partial_{\rho_j} 
	v(0_j,\eta_{\hat{j}},\xi) +
	i\eta_j \widehat{g}_{j}
	(\eta_{\hat{j}},\xi)}
{\eta^2 +\Xi_{bj}^2}
e^{i \rho \cdot \eta} e^{ix\xi} d\eta  d\xi\\
& 
\label{vPoissonBasic}
+
A_{-3}\left(\partial_{\rho}
v \big|_{\partial\Omega}
\right) +A_{-2} g+
A_{-\infty} v.
\end{align}

In the integral in \eqref{vPoissonBasic}
for $j=1$,
 we integrate with respect to
  $\eta_1$
for $\rho_1>0$.
We then let $\rho_1\rightarrow 0^+$ to obtain
\begin{align*}
\nonumber
0=& 
\frac{i}{(2\pi)^{2n-1}} \int \frac{
	F.T._{\hat{1}} \partial_{\rho_1}
	v(0_1,\eta_{2},\xi) -
	\sqrt{\eta_{2}^2+ 
		\Xi_{b1}^2}  \widehat{g}_{1}
	(\eta_{2},\xi)}
{2i\sqrt{\eta_{2}^2+ 
		\Xi_{b1}^2} }
e^{i \rho_{2} \cdot \eta_{2}} e^{ix\xi} 
d\eta_{2}  d\xi\\
\nonumber
& + 
\frac{1}{(2\pi)^{2n}} R_1 \circ \int \frac{
	F.T._{\hat{2}}  \partial_{\rho_2} 
	v(0_2,\eta_{1},\xi) +
	i\eta_2 \widehat{g}_{2}
	(\eta_{1},\xi)}
{\eta^2 +\Xi_{b2}^2}
e^{i \rho \cdot \eta} e^{ix\xi} d\eta  d\xi\\
\nonumber
&  
+
A_{-2,b1}\left(\partial_{\rho_1} 
v \big|_{\partial\Omega_1} 
\right)
+
A_{-1,b1} g_{1} 
+ R_{b1}^{-\infty}\\
&+ R_1 \circ 
A_{-3}\left(\partial_{\rho_2}
v \big|_{\partial\Omega_2} \right) +
R_1 \circ A_{-2} g_{2},
\end{align*}
where $R_{b1}^{-\infty}$ refers to smooth terms on $\partial\Omega_1$.

Note that the second term on the 
 right side above can be 
 written, according to our notation
from Section \ref{secInterHalf},
 as
$\lre_{-3/2} ^{21}(\partial_{\rho_2}v|_{\partial\Omega_2})
 + \lre_{-1/2}^{21} g_{2}$, as can the last two 
terms.  Inverting the operator with
 symbol
locally given by 
 $1/2\sqrt{\eta_2^2+\Xi_{b1}^2}$ yields
\begin{multline}
\label{bndry1}
 \partial_{\rho_1} v |_{\partial\Omega_1}
   = |D_{b1}| g_1 +
    A_{1,b1} \circ \lre_{-1/2}^{21} g_{2} 
   + A_{0, b1} g_1\\ + 
   A_{-1,b1}  (\partial_{\rho_1}v|_{\partial\Omega_1})
   +
 \lre_{-1/2}^{21} (\partial_{\rho_2}v|_{\partial\Omega_2}),
\end{multline}
where $|D_{b1}|$ is the operator with symbol 
 (in an neighborhood of the origin)
 given by
$\sqrt{\eta_2^2+\Xi_{b1}^2}$.

Similar calculations lead to the expression
\begin{multline}
\label{bndry2}
\partial_{\rho_2} v |_{\partial\Omega_2}
= |D_{b2}| g_2 +
A_{1,b2} \circ \lre_{-1/2}^{12} g_{1} 
+ A_{0,b2} g_2 + \\
A_{-1}  (\partial_{\rho_2}v|_{\partial\Omega_2})
+
\lre_{-1/2}^{12} (\partial_{\rho_1}v|_{\partial\Omega_1})
\end{multline}
for $\partial_{\rho_2} v |_{\partial\Omega_2}$.
 We note the occurrence of the $ \lre_{-1/2}^{jk}$
operators, which shows
 an expression for the Dirichlet to Neumann operator (DNO),
giving the boundary values of the normal derivative of $v$, 
 on intersection domains is not as simple as finding 
an expansion in terms of pseudodifferential operators 
 as in the case of smooth domains.
 
Before we further explore reduction to the boundary techniques, we use \eqref{bndry1} and \eqref{bndry2}
 to obtain the first few principal terms in the
Poisson operator.  Inserting \eqref{bndry1}
  and \eqref{bndry2} in 
\eqref{vPoissonBasic}, we obtain
\begin{align}
\nonumber
v
=&   \Theta^+_1 g_{1}
+  \Theta^+_2 g_{2}
+ \sum_{j, k}A_{-1} \circ \lre_{-1/2}^{kj} g_{k}
+
\sum_{j} A_{-2}
g_{j}\\
\label{vFormDNO}
&+
\sum_{j,k} A_{-2}
\circ \lre_{-1/2}^{kj} 
\left(
\partial_{\rho_k} v \big|_{\partial\Omega_k}
\right) 
+
A_{-3}\left(\partial_{\rho}
v \big|_{\partial\Omega}\right) 
+ R^{-\infty},
\end{align}
where $R^{-\infty}$ refers to smooth terms, and 
where the symbol of $\Theta_j^+$ is locally 
\begin{equation*}
 \sigma(\Theta^+_j) 
  = \frac{i}{\eta_j + i \Xi_{bj}}.
\end{equation*}
Such an expression can be used to determine the 
 mapping properties of the Poisson operator
using the weighted spaces of Section \ref{secInterHalf}.
 Thus we see the expression for the 
 Poisson operator on the intersection domain
  contains a sum
of the principal terms of the individual Poisson operators for each domain
 comprising the intersection.  This can be generalized to
  any number of intersections (with $m\le n$).  

Consider now the boundary value problem of the form
\begin{equation*}
\begin{aligned}
\Gamma v =& 0  \qquad \mbox{in } \Omega\\
\partial_{\nu} v + Bv =& 
 g \qquad \mbox{on } \partial \Omega,
\end{aligned}
\end{equation*}
where $\partial_{\nu}$ denotes the normal derivative
 ($\partial_{\rho_j}$ on $\partial \Omega_j$, $j=1,2$),
and $B$ is a tangential pseudodifferential operator.

We use the techniques above to reduce to the
 boundary.  From \eqref{bndry1} and
  \eqref{bndry2}, the boundary condition gives the 
two relations:
\begin{equation*}
 |D_{bj}| v_{bj} + Bv_{bj}+
 A_{1,bj}\circ \lre_{-1/2}^{kj} v_{bk}
 + A_{0,bj} v_{bj} + 
 A_{-1}  (\partial_{\rho_j}v|_{\partial\Omega_j})
 +
 \lre_{-1/2}^{kj} (\partial_{\rho_k}v|_{\partial\Omega_k})
=0
\end{equation*}
for $j,k=1,2$, $j\neq k$.
 Solving this system together with
  \eqref{bndry1} and \eqref{bndry2},
   would lead to a solution
for the Dirichlet problem.  The boundary solutions
 would be inserted in \eqref{vFormDNO} to obtain an
expression for $v$ on the entire domain.
In particular, if $|D_{bj}| + B$ forms
 an elliptic system weighted estimates can 
be obtained by inverting the highest order
 term, $|D_{bj}| + B$, and considering the remaining
terms as error terms (which lead to estimates in
 lower level Sobolev spaces).  One would
of course need estimates for the boundary values 
 of the normal derivatives, estimates which could be 
  obtained with the help of \eqref{bndry1} and
\eqref{bndry2}.
 
We can immediately see the need for weighted estimates 
 due to the $\lre_{-\alpha}^{jk}$ operators.  
The use of weights allows us to consider higher order
 Sobolev estimates, which would be unobtainable without
the introduction of terms which deal with the
 singularity at the points of intersection.  
 { Weights are also needed even for the 
	pseudodifferential operators, since the arguments are considered to be
	extended by zero.  For instance, 
	for a function, $u_{bj} \in W^s(\partial\Omega_j\cap\partial\Omega)$ 
	in a
	term of the form $A_{0,bj}u_{bj}$, the argument,
	$u_{bj}$ is considered to be extended by zero to all of
	$\partial\Omega_j$ and this extension is not 
	in $W^{s}(\partial\Omega)$ }. 
 
Of particular interest to the author is the case
 in which $|D_{bj}| + B$ does not form an 
elliptic system.  Such is the case in the 
 \dbar-Neumann problem on intersection domains.  
In this case, the zeroth order terms in the 
 expansion of the DNO are required, as well as 
a method to deal with the $\lre_{-1/2}^{jk}$ operators. 
 Such is the subject of a current study, which the
author will publish separately.

\end{document}